\documentclass[11pt]{article}

\usepackage{star}

\title{Decentralized Online Riemannian Optimization with Dynamic Environments}
\author{Hengchao Chen\thanks{University of Toronto 
  ({hengchao.chen@mail.utoronto.ca}).} \and Qiang Sun\thanks{Corresponding author. University of Toronto and  MBZUAI ({qsunstats@gmail.com}).}}
\date{}

\begin{document}

\maketitle

\begin{abstract}
This paper develops the first decentralized online Riemannian optimization algorithm on Hadamard manifolds. Our algorithm, the decentralized projected Riemannian gradient descent, iteratively performs local updates using  projected Riemannian gradient descent  and a consensus step via weighted \Frechet mean.  Theoretically, we establish linear variance reduction for the consensus step. Building on this, we prove a dynamic regret bound of order $\cO(\sqrt{T(1+P_T)}/\sqrt{(1-\sigma_2(W))})$, where $T$ is the time horizon, $P_T$ represents  the path variation measuring nonstationarity, and $\sigma_2(W)$ measures the network connectivity. The weighted \Frechet mean in our algorithm incurs a minimization problem, which can be computationally expensive. To further alleviate this  cost, we propose a simplified consensus step with a closed-form, replacing the weighted \Frechet mean. We then establish linear variance reduction for this alternative and prove  that the decentralized algorithm, even with this  simple consensus step,  achieves the same dynamic regret bound. 
Finally, we validate our approach with experiments on nonstationary decentralized \Frechet mean computation over hyperbolic spaces and the space of symmetric positive definite matrices, demonstrating the effectiveness of our methods.



\end{abstract}





\section{Introduction}\label{sec:introduction}

Distributed optimization has gained widespread attention from researchers across various  fields in science and engineering \cite{yang2019survey}. Its applications span numerous domains, including  power systems \cite{wang2017distributed,molzahn2017survey}, sensor networks \cite{rabbat2004distributed}, and smart manufacturing \cite{kusiak2018smart}, among others. The objective  of distributed optimization is to minimize a global  function $f$, which is the sum of local functions, each accessible to individual agents within a network. Each  agent can either update its decision variable using local information or exchange information with other agents through the network. This decentralized approach benefits from  reduced  computational burden, lower  memory requirements, improved privacy, and greater resilience to node failures compared to centralized systems.

The core components of these algorithms are local updates and consensus steps. Local updates allow agents to make decisions based on their own information, while consensus steps ensure that agents reach agreement on their local estimates, facilitating collective optimization of a global function. 
For Euclidean cases, this can be achieved  by the weighted average:
\#\label{equ:1}
x_{i,t+1}=\sum_{j=1}^nw_{ij}x_{j,t},\quad i=1,\ldots,n,
\#
where $x_{i,t}$ denotes the state of agent $i$ at time $t\geq 1$
, $w_{ij}$ represents  the weight, and $w_{ij}>0$ indicates a connection between agents $i$ and $j$. The matrix $W=(w_{ij})\in\RR^{n\times n}$ is assumed to be symmetric and doubly stochastic, satisfying  $\sum_iw_{ij}=\sum_{j}w_{ij}=1$ and $w_{ij}\geq 0$. Under these conditions, the agents' states converge linearly to a common state {as $t$ approaches infinity} if $\sigma_2(W)<1$,
where $\sigma_2(W)$ denotes the second largest singular value of $W$. Building on this consensus mechanism, various decentralized optimization algorithms have been developed to minimize global functions, including decentralized subgradient methods \cite{chen2021distributed,tsitsiklis1986distributed,yuan2016convergence} and decentralized gradient tracking methods \cite{qu2017harnessing,nedic2017achieving,shi2015extra}.

The decentralized algorithms mentioned above rely heavily on Euclidean assumptions. However, many real-world problems involve optimizing functions on Riemannian manifolds, such as symmetric positive definite matrices \cite{sra2015conic} and hyperbolic spaces \cite{nickel2017poincare,krioukov2010hyperbolic}. The inherent non-Euclidean structures of these manifolds pose significant challenges for both Riemannian optimization and  decentralized optimization.
Notably, the consensus step \eqref{equ:1} is not directly applicable due to the lack of vector space structures on manifolds.
Additionally, when local functions are defined on manifolds,  optimization algorithms must incorporate geometric structures, adding further complexity to the theoretical analysis.
This paper aims to address these challenges by developing the first decentralized Riemannian optimization algorithm over Hadamard manifolds, backed by strong theoretical guarantees.


In this paper, we extend geodesically convex optimization over Hadamard manifolds \cite{zhang2016first} to a decentralized setting. While several works have explored decentralized Riemannian optimization \cite{chen2023local,chen2024decentralized,hu2023achieving,deng2023decentralized,wang2022decentralized,sun2024global,chen2021decentralized,deng2023decentralizeddou,hu2023decentralized,wu2023decentralized}, they primarily focus on non-geodesically convex optimization on compact manifolds such as Stiefel manifolds. Therefore, these works can only establish convergence to stationary points. In sharp contrast, our work addresses Hadamard manifolds, including hyperbolic spaces and the space of symmetric positive definite matrices as examples \cite{chen2024riemannian}. Assuming the objective functions are geodesically convex, we provide a dynamic regret analysis in analogy to online convex optimization in  Euclidean spaces \cite{cesa2006prediction,besbes2015non}.  In particular, we can select the comparator as the global minimizer, rather than a stationary point, of the objective function.

More precisely, we consider a setting with $n$ agents, where each agent receives a sequence of geodesically convex functions $f_{i,t}$ defined on a Hadamard manifold $\cM$ in an online manner. The goal is to minimize the global function 
\#\label{equ:global-function}
f_t(\cdot)=\frac{1}{n}\sum_{i=1}^nf_{i,t}(\cdot).
\#
Similar to decentralized online convex optimization \cite{shahrampour2017distributed}, we  maintain the decision variables $\{x_{i,t}\}$ for all agents to minimize regret. In stationary cases, the objective is to minimize the static regret with respect to the best decision in hindsight:
\$
\textnormal{Regret}(\{x_{i,t}\})=\frac{1}{n}\sum_{t=1}^T\sum_{i=1}^nf_t(x_{i,t})-\min_x\sum_{t=1}^Tf_t(x).
\$
In dynamic environments, where the optimal decision varies over time \cite{lu2018learning,besbes2015non,zhang2018adaptive,jadbabaie2015online,yang2016tracking,zhao2020dynamic}, a more appropriate criterion is dynamic regret, defined as
\#\label{equ:1.2}
\textnormal{D-Regret}(\{x_{i,t}\})=\frac{1}{n}\sum_{t=1}^T\sum_{i=1}^nf_t(x_{i,t})-\sum_{t=1}^Tf_t(u_t),
\#
where $\{u_t\}_{t=1}^T$ are the time-varying comparators. To minimize dynamic regret, we propose a  decentralized projected Riemannian gradient descent algorithm.
This method alternatively performs local updates using projected Riemannian gradient descent and a  Riemannian consensus step. In the consensus step, we use weighted \Frechet mean to replace the Euclidean consensus step \eqref{equ:1}.
Theoretically, we show that this  consensus step achieves linear variance contraction. Building on this,  we derive a dynamic regret bound of order  $\cO(\sqrt{T(1+P_T)}/\sqrt{1-\sigma_2(W)})$, where $T$ is the time horizon, $\sigma_2(W)$ is the second largest singular value of $W$, measuring network connectivity, $P_T=\sum_{t}d(u_{t+1},u_t)$ 
measures the path variation, and $d(\cdot,\cdot)$ denotes the geodesic distance.  This  bound exactly matches  the Euclidean decentralized case \cite{shahrampour2017distributed}. In the special cases where $P_T=0$ or $\sigma_2(W)$, our results correspond to stationary or centralized scenarios, highlighting the generality of our {results}.






One limitation of the above base algorithm  is that the consensus step requires solving a minimization  problem to compute the weighted \Frechet mean. To improve computational efficiency, we propose to use a single-step Riemannian gradient descent step  as the consensus step. Perhaps surprisingly, we establish  linear variance contraction for this simple consensus step and establish the same dynamic regret bound for the decentralized algorithm.

Finally, we conduct numerical experiments on nonstationary  decentralized \Frechet mean computation over two types of Hadamard manifolds: hyperbolic spaces and the space of symmetric positive definite matrices. The results strongly support the effectiveness of proposed method, and the code is available at \href{https://github.com/statsle/DPRGD}{https://github.com/statsle/DPRGD}.

\subsection{Related Literature}

This work lies at the intersection of three key research areas: Riemannian optimization, decentralized optimization, and online optimization in dynamic environments. In this section, we provide a brief review of the literature in these fields. 

\paragraph{Riemannian Optimization}

Riemannian optimization provides a flexible framework for optimizing a function defined on non-Euclidean spaces, such as symmetric positive definite matrices, hyperbolic spaces,  homogeneous spaces, and stratified spaces \cite{chen2024riemannian,absil2008optimization,hu2020brief,chen2024quotient}. Leveraging various geometric tools, this framework can efficiently address complex problems in various domain such as computer vision and robotics. Similar to Euclidean optimization, Riemannian optimization can be divided into two main categories: geodesically convex optimization and non-geodesically convex optimization.

In the geodesically convex regime, \cite{zhang2016first} established the first global convergence rate of several first-order methods,  including Riemannian gradient descent. Their result applies to Hadamard manifolds, where many problems,  such as  \Frechet mean computation,  are geodesically convex \cite{karcher1977riemannian,afsari2011riemannian,bacak2014convex}. Building on this, \cite{liu2017accelerated} extended Nesterov’s accelerated methods to Riemannian manifolds, developing accelerated first-order methods for geodesically convex optimization. Subsequent works \cite{zhang2018estimate,ahn2020nesterov,alimisis2021momentum,jin2022understanding,kim2022accelerated,martinez2023accelerated,han2023riemannian} have expanded on these methods, improving computational efficiency and relaxing required assumptions.

In contrast, the non-geodesically convex regime deals with general Riemannian manifolds and objective functions. Many first- and second-order optimization methods have been developed for these settings \cite{boumal2023introduction,absil2008optimization,boumal2019global,chen2020proximal,li2021weakly,bonnabel2013stochastic}, finding practical applications in tasks such as robust subspace recovery and orthogonal dictionary learning. However, due to non-convexity, solution iterates are only guaranteed to converge to  stationary points rather than an optimal solution. This limitation makes regret analysis, particularly in nonstationary environments, more challenging. This paper focuses exclusively on geodesically convex optimization.


\paragraph{Decentralized Optimization}

Research on decentralized optimization in Euclidean spaces is well-developed \cite{nedic2018network}, with notable methods including decentralized subgradient methods \cite{chen2021distributed,tsitsiklis1986distributed,yuan2016convergence}, primal-dual approaches \cite{duchi2011dual,alghunaim2020decentralized,xiao2010dual}, gradient tracking algorithms \cite{qu2017harnessing,nedic2017achieving,shi2015extra}, and augmented distributed methods \cite{aybat2017distributed,chang2014multi,xu2015augmented}. The core components of these algorithms are local updates and consensus steps. 
Extending these concepts to nonlinear spaces is nontrivial  due to the absence of vector space structures and the need to incorporate geometric structures specific to the underlying manifold. 

Early works on consensus for Riemannian manifolds \cite{tron2012riemannian,sarlette2009consensus,markdahl2017almost,markdahl2020synchronization,markdahl2020high} primarily focus on establishing asymptotic convergence of agents' states to consensus.  Building on these works, \cite{mishra2019riemannian} develops a decentralized optimization algorithm on the Grassmann manifold, which guarantees asymptotic convergence but does not provide explicit convergence rate guarantees. To address this gap, recent studies \cite{chen2023local,chen2024decentralized,hu2023achieving,deng2023decentralized} establish local linear convergence rates for various consensus algorithms on compact submanifolds, such as the Stiefel manifold. These advancements have led to the development of new decentralized optimization methods, including decentralized Riemannian gradient descent and gradient tracking \cite{chen2021decentralized}, augmented Lagrangian methods \cite{wang2022decentralized}, decentralized projected gradient methods \cite{deng2023decentralized}, decentralized Douglas-Rachford splitting methods \cite{deng2023decentralizeddou}, decentralized Riemannian natural gradient descent \cite{hu2023decentralized}, decentralized Riemannian conjugate gradient methods \cite{chen2024decentralized}, and retraction-free methods \cite{sun2024global}. However, these works focus on nonconvex optimization, resulting in convergence to stationary points rather than global optima.

To our knowledge, decentralized geodesically convex optimization over Hadamard manifolds remains an unexplored area. Recent studies on federated learning over general Riemannian manifolds \cite{li2023federated,huang2024federated} rely on a central server, which deviates from a fully decentralized framework.

\paragraph{Online Optimization in Dynamic Environments}

Online learning considers a sequence of functions $\{f_t\}$  and updates the decision  $x_t$  to minimize regret \cite{cesa2006prediction}:
\$
\textnormal{Regret}(\{x_t\})=\sum_{t=1}^Tf_t(x_t)-\min_x\sum_{t=1}^Tf_t(x).
\$
The above regret assumes that the comparator remains static over time, which is often inadequate in dynamic environments. To handle nonstationarity, many works have examined the following dynamic regret \cite{besbes2015non,zhang2018adaptive,jadbabaie2015online,yang2016tracking,zhao2020dynamic}:
\$
\textnormal{D-Regret}(\{x_t\})=\sum_{t=1}^Tf_t(x_t)-\sum_{t=1}^Tf_t(u_t),
\$
where $\{u_t\}$ are time-varying comparators. Using the path variation, $P_T=\sum_t\norm{u_{t+1}-u_t}$,  to quantify nonstationarity, the minimax optimal dynamic regret bound is known to be 
$\Theta(\sqrt{T(1+P_T)})$ \cite{zhang2018adaptive}.

Extensions of online learning include online Riemannian optimization \cite{hu2024riemannian,wang2023online,hu2023minimizing,wang2023riemannian} and nonstationary distributed online learning \cite{shahrampour2017distributed,yi2020distributed,li2020distributed}. To the best of our knowledge, no prior work has addressed decentralized online Riemannian optimization in dynamic environments, making this the first study in that direction. Furthermore, our work is also the first to investigate decentralized geodesically convex optimization, even in the static case.

\subsection{Overview}

The rest of this paper proceeds as follows. In Section \ref{sec:2}, we introduce the preliminaries, problem formulation, algorithm, and discuss the basic assumptions and properties. Section \ref{sec:3} presents a theoretical analysis of the proposed decentralized method. In Section \ref{sec:opt-free}, we introduce  an improved decentralized algorithm using  a simple  consensus mechanism and derive its dynamic regret bound. Numerical experiments are conducted in Section \ref{sec:numerical}. Section \ref{sec:conclusion} concludes the paper  and the appendix collects proofs.  


\section{Problem Formulation and Algorithm}\label{sec:2}

In this section, we begin by reviewing key geometric concepts and then formulate the decentralized optimization problem in Section \ref{sec:2.2}. Subsequently, we propose our decentralized algorithm in Section \ref{sec:2.3} and discuss its assumptions and relevant properties in Section \ref{sec:2.4}.

\subsection{Preliminary}

We first  recap key concepts from Riemannian geometry and function classes defined on a Hadamard manifold. For more details  on Riemannian geometry, we refer readers to \cite{petersen2006riemannian,chen2024riemannian,absil2008optimization}. A Riemannian manifold $(\cM,g)$ is a smooth manifold equipped with a Riemannian metric,  a smoothly varying family of inner products $g_x$ defined on the tangent space at each point  $x \in \cM$. A curve in $\cM$ is called a geodesic if it locally minimizes\footnote{We say that a curve $\gamma: I \mapsto M$ is locally minimizing if any $t_0 \in I$ has a neighborhood
$U\subset I$ such that $\gamma$
is minimizing the length between each pair of the points within $U$.} 
the length between points. For any $x\in\cM$ and $v\in T_{x}\cM$, there exists a unique geodesic $\gamma_v(t)$ with $x$ as its initial position and $v$ as its initial velocity. The exponential map $\Exp_x:T_{x}\cM\to\cM$ is defined by $\Exp_x(v)=\gamma_v(1)$, provided $\gamma_v(1)$ exists. The distance between two points is defined as the minimum length over all piecewise smooth curves connecting them.

A Hadamard manifold is a simply connected, complete Riemannian manifold with nonpositive curvature. For any two points on a Hadamard manifold, there is a unique geodesic connecting them. Additionally, for any $x\in\cM$, the exponential map $\Exp_x$ is a diffeomorphism between $T_x\cM$ and $\cM$, and its inverse is called the logarithm map, denoted by $\Log_x$.  

Let $\cM$ be a Hadamard manifold. A set $\cX\subseteq\cM$ is  geodesically convex if, for any $x,y\in\cX$, the geodesic connecting $x$ and $y$ lies entirely within $\cX$. A function  $f:\cM\to\RR$ is said to be geodesically convex if,  for any $x,y\in\cM$, any geodesic $\gamma$ with $\gamma(0)=x$ and $\gamma(1)=y$, and for any $t\in[0,1]$, it holds that 
\$
f(\gamma(t))\leq (1-t)f(x)+tf(y).
\$
Let $f$ be geodesically convex and differentiable. Then the gradient $\nabla f(x)\in T_x\cM$ of $f$ at $x$ is a tangent vector such that 
\$
f(y)\geq f(x)+\inner{\nabla f(x)}{\Log_x(y)},
\$
where $\inner{\cdot}{\cdot}$ denotes the inner product in $T_x\cM$. Moreover, $f$ is  $L$-Lipschitz if for any $x,y\in\cM$, 
\$
|f(x)-f(y)|\leq Ld(x,y),
\$
where $d(\cdot,\cdot)$ denotes the geodesic distance. Throughout the rest of this paper, we assume $\cM$ is a Hadamard manifold and $f$ is a geodesically convex, differentiable, and $L$-Lipschitz function  defined on $\cM$.

\subsection{Problem Formulation}\label{sec:2.2}

Consider a network with $n$ agents. At each time $t$, agent $i$ receives a geodesically convex, differentiable, and $L$-Lipschitz\footnote{Our methods and theory also apply to  cases where $f_{i,t}$ are $L$-Lipschitz over a bounded region $\cX^{D}=\{x\mid d(x,\cX)\leq D\}$ without modifications.} function $f_{i,t}$. The global function is given as  $f_t=\sum_if_{i,t}/n$. Assume that the global function $f_t$ has a unique minimizer $x_t^*$ in $\cM$, which lies  in a fixed, bounded,  closed,  geodesically convex set $\cX\subseteq\cM$\footnote{This is not a essential assumption for our algorithm and dynamic regret analysis, and we make this assumption to allow the choice of unique minimizers as comparators.
}. The diameter of $\cX$ is denoted by $D=\sup_{x,y\in\cX}d(x,y)$.

At time $t$, each agent updates its state $x_{i,t}$, and exchanges information with other agents through the network. Inspired  by the consensus step \eqref{equ:1} in Euclidean spaces, we consider  the following consensus step:\footnote{In Section \ref{sec:opt-free}, we provide a simple consensus step \eqref{equ:2.1} which further improves computational efficiency. Comparing with \eqref{equ:2.1}, this simple consensus step avoids additional tuning  and offers a mathematically cleaner formulation.} 
\#\label{equ:2.1}
y_{i,t+1}=\argmin_{y}\sum_{j=1}^nw_{ij}d^2(y,y_{j,t}),\quad i=1,\ldots,n,
\#
where the weight matrix $W=(w_{ij})$ is symmetric and doubly stochastic, and $\{y_{i,t}\}$ and $\{y_{i,t+1}\}$ denote the input and output variables. This consensus step outputs a weighted \Frechet mean, which exists and is unique in a Hadamard manifold. Moreover, if $\{y_{i,t}\}\subseteq\cX$ for a geodesically convex set $\cX$, then $\{y_{i,t+1}\}\subseteq\cX$ as well  \cite{sturm2003probability}. In Euclidean spaces, this reduces to \eqref{equ:1}. Notably, it is a decentralized method, as at each time $t$, agent $i$ only receives information $\{y_{j,t}\}$ from connected agents ($j$ such that $w_{ij}>0$).

However, a limitation of the update \eqref{equ:2.1} is that it does not admit a closed-form solution and requires an additional optimization algorithm. Fortunately, when $\cM$ is a Hadamard manifold, the optimization problem in \eqref{equ:2.1} is well studied \cite{bacak2014computing}. {For now, we assume that the minimizer in \eqref{equ:2.1} is attainable.}
Our objective is to develop an algorithm that minimizes the dynamic regret \eqref{equ:1.2}, where the comparators $\{u_t\}$ are given by the global minimizers $\{x_t^*\}$ or another sequence of points in $\cX$.


\subsection{A Decentralized Optimization Algorithm}\label{sec:2.3}

Based on the consensus step \eqref{equ:2.1}, we  develop a decentralized optimization algorithm for minimizing the global function. Our strategy alternates between projected Riemannian gradient descent and the consensus step \eqref{equ:2.1}. Specifically, we initialize $x_{i,1}=x_1\in\cX$ for all agents $i$. At time $t$, after receiving the function $f_{i,t}$, each agent performs the following updates:
\#
y_{i,t+1}&=\cP_{\cX}(\Exp_{x_{i,t}}(-\eta \nabla_{i,t})),\quad \nabla_{i,t}=\nabla f_{i,t}(x_{i,t}),\tag{I}\label{alg:I}\\
x_{i,t+1}&=\argmin_{x}\sum_{j=1}^nw_{ij}d^2(x,y_{j,t+1}), \tag{II}\label{alg:II}
\#
where $\eta>0$ is the stepsize, and $\cP_{\cX}(z)=\argmin_{y\in\cX}d(y,z)$ is the projection operator onto the geodesically convex set $\cX$.  The resulting algorithm is  called {\bf decentralized projected Riemannian gradient descent} ({\bf DPRGD}), where step \eqref{alg:I}  corresponds to  projected Riemannian gradient descent, and step \eqref{alg:II} corresponds to  the consensus step. 
While we use  the same name, our algorithm is distinct from the one proposed in \cite{deng2023decentralized}; we focus on geodesically convex optimization on Hadamard manifolds, whereas their work examines non-geodesically convex optimization on compact manifolds.  
Our goal is to derive upper bounds for the dynamic regret as defined in \eqref{equ:1.2}.

Before proceeding to theoretical analysis, we highlight that our decentralized algorithm resembles the decentralized online mirror descent proposed in Euclidean spaces \cite{shahrampour2017distributed}. Specifically, our projected Riemannian gradient descent corresponds to their mirror descent step when the Bregman divergence is chosen as the quadratic function. Additionally, our consensus step reduces to the weighted average \eqref{equ:1}, which is the Euclidean consensus step used in their algorithm. Thus our algorithm is a generalized version of their algorithm in the Riemannian manifold setting.

This generalization presents two key challenges. First, the consensus step \eqref{equ:2.1} is not a simple matrix multiplication, introducing new difficulties in the consensus analysis. Second, the regret analysis for Riemannian optimization requires advanced geometric techniques, such as comparison theorems from Riemannian geometry \cite{zhang2016first,petersen2006riemannian}.


\subsection{Assumptions and Properties}\label{sec:2.4}

This section outlines several assumptions and properties relevant the decentralized problem.
The first assumption concerns the connectivity of the network. We assume that  $W$  is  symmetric and  doubly stochastic which can be easily constructed from an undirected connected network using the Metropolis rule, the Laplacian-based constant edge weight matrix  \cite{shi2015extra}, or the maximum-degree rule \cite{boyd2004fastest}.  The second-largest singular value  $\sigma_2(W)$  of $W$  characterizes the network’s connectivity, with a smaller  $\sigma_2(W)$  indicating better connectivity. To ensure the feasibility of the decentralized algorithm, we assume  $\sigma_2(W) < 1$. In many decentralized problems, the convergence rate or regret analysis depends positively on  $1/(1 - \sigma_2(W))$  \cite{shahrampour2017distributed,nedic2018network},  which is assumed to be finite.



Second, we assume that the comparators  $\{u_t\}_{t=1}^n$  lie within  $\cX$, a closed, bounded, geodesically convex subset of $\cM$. Bounded decision variables are a standard requirement in online convex optimization \cite{cesa2006prediction,shahrampour2017distributed,besbes2015non}, ensuring that  $d^2(x, y)$ is Lipschitz over  $\cX$ . This corresponds to the Lipschitz property of the Bregman divergence, as noted in \cite[Assumption 4]{shahrampour2017distributed}, and is also assumed in the original paper \cite{zhang2016first}  discussing geodesically convex optimization. One useful implication of this assumption is the following lemma concerning the projection  $\cP_{\cX}$ onto the set $\cX$.  

\begin{lemma}[Lemma 7, \cite{zhang2016first}, \cite{bacak2014computing}]\label{lma:1}
    Let $\cM$ be a Hadamard manifold and $\cX\subseteq\cM$ be a closed geodesically convex set. Then the projection $\cP_{\cX}(z)=\argmin_{y}d(y,z)$ is single-valued and nonexpansive.  That is, for any $x,y\in\cM$, 
    \$
    d(\cP_{\cX}(x),\cP_{\cX}(y))\leq d(x,y).
    \$
    In particular, $d(x,\cP_{\cX}(y))\leq d(x,y)$ for all $x\in\cX$ and $y\in\cM$.
\end{lemma}

Using Lemma \ref{lma:1} and the definition of the exponential map,  we can show that the projected Riemannian gradient descent step \eqref{alg:I} produces $y_{i,t+1}$ such that $d(y_{i,t+1},x_{i,t})\leq \eta\norm{\nabla_{i,t}}$. Here we use the fact that $x_{i,t}\in\cX$ as a consequence of the properties of the weighted \Frechet mean on Hadamard manifolds \cite[Proposition 6.1]{sturm2003probability}. 



Furthermore, we assume that  $\cM$  is a Hadamard manifold with sectional curvatures bounded below by some  $\kappa < 0$. This assumption aids in the non-asymptotic convergence analysis of first-order methods for geodesically convex optimization \cite{zhang2016first}. A useful lemma under this assumption is provided below, which allows for the comparison of squared distances.

\begin{lemma}[Lemma 5, \cite{zhang2016first}]\label{lma:2}
    Suppose $\cM$ is a Hadamard manifold with sectional curvatures lower bounded by $\kappa<0$. For any geodesic triangle in $\cM$ with edge lengths $a,b,c$, it holds that
    \$
    a^2\leq  \sqrt{|\kappa|}c\coth(\sqrt{|\kappa|}c)b^2 + c^2 - 2bc\cos A,
    \$
    where $A$ is the opposite angle associated with $a$. 
\end{lemma}

\section{Theoretical Analysis}\label{sec:3}

In this section, we present a dynamic regret analysis for decentralized projected Riemannian gradient descent. Our theoretical framework resembles that of \cite{shahrampour2017distributed},  where they upper bound the tracking error and network error. To generalize their results to nonlinear settings, we leverage geometric tools to develop novel theories. 
Specifically, in Section \ref{sec:3.1}, we first analyze the consensus step \eqref{equ:2.1} or \eqref{alg:II}. We show that the weighted \Frechet mean-based consensus exhibits a linear contraction rate determined by  $\sigma_2(W)$. Building on this, we derive the tracking and network errors in Section \ref{sec:3.2}. Finally, by combining these errors, we establish an upper bound on the dynamic regret and discuss its implications in both static and centralized scenarios.


\subsection{Consensus}\label{sec:3.1}

Our first result focuses on the consensus step \eqref{alg:II}. Given a  matrix $W=(w_{ij})$, the consensus step  can be viewed as a mapping from the set of points  $\{y_{i,t+1}\}$  to a new set of points  $\{x_{i,t+1}\}$. In the following lemma, we show that this map significantly reduces the \Frechet variance of the data points. As a result, this implies that the agents’ states will converge to consensus at a linear rate if we repeatedly apply the consensus step \eqref{alg:II} (without performing  \eqref{alg:I}). This lemma is crucial for quantifying the network error in Section \ref{sec:3.2}. 



\begin{lemma}[Linear variance reduction for  consensus step \eqref{alg:II}]\label{lma:3}
    Suppose $\cM$ is a Hadamard manifold. Suppose $\{y_{i,t+1}\}_{i=1}^n$ are $n$ points on the manifold and $\{x_{i,t+1}\}_{i=1}^n$ are the transformed $n$ points via \eqref{alg:II}, where $W=(w_{ij})$ is a symmetric doubly stochastic matrix. Then it holds that
    \#\label{equ:contraction}
\VF(\{x_{i,t+1}\})\leq \sigma_2^2(W)\cdot \VF(\{y_{i,t+1}\}),
\#
where $\VF(\{z_i\})$ denotes the \Frechet variance of $\{z_i\}_{i=1}^n$ defined by 
\#\label{equ:VF}
\VF(\{z_i\})=\min_z\frac{1}{n}\sum_id^2(z,z_i).
\#
\end{lemma}

\subsection{Network and Tracking Errors}\label{sec:3.2}


We now proceed to investigate the decentralized algorithm. The first lemma provides an upper bound on the deviation of the local state from the consensus state at each iteration, where the bound is proportional to the stepsize  $\eta$  in \eqref{alg:I}. The proof relies on Lemma \ref{lma:3} and is deferred to the appendix.

\begin{lemma}[Network error]\label{lma:4}
    Consider the problem setting in Section \ref{sec:2.2} and suppose conditions in Section \ref{sec:2.3} hold. Let $\{x_{i,t}\}$ be the iterates  from the decentralized algorithm in Section \ref{sec:2.3}. Then we have
    \$
    d(x_{i,t},\bar x_t)\leq \frac{\eta\sigma_2(W)\sqrt{n}L}{1-\sigma_2(W)},\quad\forall i,t,
    \$
    where $\bar x_t$ is the \Frechet mean of $\{x_{i,t}\}_{i=1}^n$ and $\sigma_2(W)<1$ is the second largest singular value of $W$. 
\end{lemma}


The upper bound in Lemma \ref{lma:4} is similar to that established in Lemma 1 of \cite{shahrampour2017distributed}. This bound is fundamentally linked to the network connectivity, represented by  $\sigma_2(W)$. When  $\sigma_2(W) = 0$, the network achieves perfect consensus among its agents. Conversely, as  $\sigma_2(W)$  approaches 1, the disparity between node states can grow without limit in the worst-case scenario. This relationship highlights how network topology influences decentralized optimization \cite{nedic2018network}. Additionally, the upper bound shows a linear dependence on the stepsize  $\eta$. A diminishing stepsize leads to asymptotic consensus among the agents’ states. However, as we will demonstrate, choosing an excessively small stepsize may result in increased optimization error. Therefore, selecting an appropriate stepsize is crucial for achieving optimal performance. To further explore this relationship, we next analyze the tracking error in the following lemma, with the proof provided in the appendix.

\begin{lemma}[Tracking error]\label{lma:5}
    Consider the problem setting in Section \ref{sec:2.2} and assume the conditions in Section \ref{sec:2.4} hold. Let $\{x_{i,t},y_{i,t+1}\}$ be the iterates generated by the decentralized algorithm in Section \ref{sec:2.3}, and $\{u_t\}$ be a sequence of points in the convex set $\cX$. Then it holds that
    \$
    \frac{1}{n}\sum_{i=1}^n\sum_{t=1}^T\left(d^2(u_t,x_{i,t})-d^2(u_t,y_{i,t+1})\right) 
    \leq D^2+2D\sum_{t=1}^Td(u_{t+1},u_t),
    \$
    where $D=\sup_{x,y\in\cX}d(x,y)$ is the diameter of $\cX$. 
\end{lemma}


Lemma \ref{lma:5}  analyzes the projected Riemannian gradient descent step \eqref{alg:I}. The term  $P_T = \sum_{i=1}^T d(u_{t+1}, u_t)$  in the upper bound is analogous to the path variation in Euclidean scenarios \cite{zhang2018adaptive}, and it quantifies the nonstationarity of the comparators  $\{u_t\}$.   Achieving sublinear dynamic regret requires  $P_T = o(T)$  \cite{zhang2018adaptive}.

It is important to note that the comparators  $\{u_t\}$ in Lemma \ref{lma:5} can be arbitrary.  While it is common to select the minimizers of the global functions  $\{f_t\}$ as comparators, this can result in a large path variation $P_T$. Instead, it may be more advantageous to choose a sequence of points in  $\mathcal{X}$  that balances the path variation $P_T$ and the comparator performance $\sum_tf_t(u_t)$. In the Euclidean setting, 
optimizing over comparators in this manner  often leads to improved regret bounds \cite{zhang2018adaptive}. 


\subsection{Finite-Time Performance: Dynamic Regret Bounds}\label{sec:3.3}

This section combines Lemmas \ref{lma:4} and \ref{lma:5} to derive an upper bound on the dynamic regret of the decentralized algorithm in Section \ref{sec:2.3}. As in Lemma \ref{lma:5}, our theory accommodates any sequence of comparators $\{u_t\}$ in $\cX$. The main result is presented in the following theorem.

\begin{theorem}[Dynamic regret]\label{thm:1}
Consider the problem setting in Section \ref{sec:2.2} and assume conditions in Section \ref{sec:2.4} hold. Let $\{x_{i,t}\}$ be the iterates generated by the decentralized algorithm in Section \ref{sec:2.3}, and let $\{u_t\}$ be a sequence of points in the convex set $\cX$. Then, the dynamic regret \eqref{equ:1.2} can be bounded as follows:
\$
\textnormal{D-Regret}(\{x_{i,t}\})\leq \eta TC+\frac{1}{\eta}(D^2+2DP_T),
\$
where $C=L^2\zeta(\kappa,D+\eta L)+2\sigma_2(W)\sqrt{n}L^2/(1-\sigma_2(W))$ is a constant, $\zeta(\kappa,c)=\sqrt{|\kappa|c}\coth(\sqrt{|\kappa|}c)$, $\kappa<0$ is the lower bound on the sectional curvature of $\cM$, $D$ is the diameter of $\cX$,  $L$ is the Lipschitz constant of $f_{i,t}$, and $P_T=\sum_{t=1}^Td(u_{t+1},u_t)$ is the path variation of the comparators $\{u_t\}$. 
\end{theorem}

The upper bounds consists of two components. The term $\eta TC$ reflects the network connectivity, while the term $(D^2+2DP_T)/\eta$ captures the nonstationarity $P_T$ of the comparators. As $\eta$ decreases, the first term reduces, but the second term increases, illustrating  the trade-off between achieving better consensus (smaller $\eta$) and managing nonstationarity (larger $\eta$). To balance this trade-off, we can choose a suitable stepsize, leading to the following corollary.

\begin{corollary}\label{corollary:1}
    Assume the same conditions as in Theorem \ref{thm:1} hold. Let the stepsize $\eta=\sqrt{D^2+2DP_T}/\sqrt{TC'}\leq 1$, where $C'$ is the constant $C$ defined in Theorem \ref{thm:1} with $\eta=1$. Then we  achieve the following dynamic regret bound:
    \$
    \textnormal{D-Regret}(\{x_{i,t}\})= \cO\left(\sqrt{\frac{T(1+P_T)}{1-\sigma_2(W)}}\right),
    \$
    which is sublinear when $P_T=o(T)$.
\end{corollary}

\begin{proof}
    This corollary immediately follows from Theorem \ref{thm:1} by substituting the stepsize into the regret bound. 
\end{proof}

When the manifold $\cM$ is a Euclidean space, our dynamic regret bound recovers the result in \cite{shahrampour2017distributed}. Notably, $1-\sigma_2(W)$ and $P_T$ characterize how network topology and nonstationarity affect the regret bound. By considering specific settings, our results reduce to the following conclusions:
\begin{enumerate}
\item {\bf Stationary Setting:} When $P_T=0$, this corresponds to the stationary setting,  and our result implies a regret bound of $\cO(\sqrt{T/(1-\sigma_2(W))})$. This is the first decentralized online geodesically convex optimization algorithm with a sublinear regret.

\item {\bf Centralized Scenario:} When $W=(1/n)$ such that $\sigma_2(W)=0$, the problem reduces to the centralized scenario. Our results then recover the dynamic regret bound $\cO(\sqrt{T(1+P_T)})$, as established in \cite{hu2023minimizing}. This regret bound  is minimax optimal,  as shown in \cite[Theorem 3]{hu2023minimizing}.  

\end{enumerate}

\subsection{Limitations}\label{sec:3.4}


Our decentralized projected Riemannian gradient descent algorithm has several limitations. First, as mentioned in Sections \ref{sec:2.2} and \ref{sec:2.3}, we assume that the weighted \Frechet mean problem \eqref{alg:II} are  efficiently solvable and the exact minimizers can be obtained. While we can use the optimization algorithm in \cite{bacak2014computing} or Riemannian gradient descent for this computation, it remains computationally expensive per iteration. To address this issue, we develop a computationally efficient consensus mechanism in Section \ref{sec:opt-free}. Second, it is also assumed that the projector $\cP_{\cX}$ in \eqref{alg:I} is efficiently solvable. This may exhibit the similar computational challenge as in the consensus step \eqref{alg:II}. 
Third, in Corollary \ref{corollary:1}, the optimal stepsize is contingent on the nonstationarity $P_T$, which is generally unavailable in practice. Therefore, it is crucial to develop an adaptive algorithm to achieve the optimal rate without prior knowledge of $P_T$ \cite{zhang2018adaptive}. 
However, even in Euclidean spaces, designing such an adaptive algorithm remains an open problem \cite{shahrampour2017distributed}, as existing approaches require knowledge of the global functions \cite{zhang2018adaptive,hu2023minimizing}, which is unavailable in decentralized settings. We leave this challenge for future study.
\section{A Decentralized Algorithm with a Computationally Efficient Consensus Step}\label{sec:opt-free}



In this section, we address the first limitation in Section \ref{sec:3.4} by introducing a computationally efficient consensus step with a closed-form update.  Given a weight matrix $W=(w_{ij})$ and a set of points $\{y_{i,t+1}\}_{i=1}^n$, our consensus step proceeds as follows:
\#\label{equ:4.1}
x_{i,t+1}=\Exp_{y_{i,t+1}}\left(\gamma\sum_{j=1}^nw_{ij}\Log_{y_{i,t+1}}y_{j,t+1}\right),\quad\forall i=1,\ldots,n,
\#
where $\gamma>0$ is the stepsize. The above consensus step first maps $\{y_{j,t+1}\}$ to the tangent space $T_{y_{i,t+1}}\cM$ at $y_{i,t+1}$, then computes the weighted average of these tangent vectors, and finally maps the result back to the manifold. It provides  a first-order approximation to the weighted \Frechet mean in \eqref{alg:II}. Additionally, the decentralized nature is preserved, as each agent $i$ only receives information from connected agents, those $j$ for which $w_{ij}>0$. Furthermore, we show that this one-step consensus step also exhibits linear variance reduction with an appropriate choice of $\gamma$. This is formally presented in the following lemma.

\begin{lemma}[Linear variance contraction in consensus step \eqref{equ:4.1}]\label{lma:4.1} 
Let $\cM$ be a Hadamard manifold with sectional curvatures lower bounded by some $\kappa<0$. Let $\{y_{i,t+1}\}_{i=1}^n$ be $n$ points in a bounded convex region $\cX\subseteq\cM$ and $\{x_{i,t+1}\}_{i=1}^n$ be the transformed $n$ points via \eqref{equ:4.1}, where $W=(w_{ij})$ is a symmetric doubly stochastic matrix. If  $\gamma=1/(2C_3)$ with $C_3=\zeta(\kappa,2D)$ and $D=\sup_{x,y\in\cX}d(x,y)$, then we have
\$
\VF(\{x_{i,t+1}\})\leq \left(1-\frac{1-\sigma_2(W)}{2C_3}\right)\VF(\{y_{i,t+1}\}),
\$
where $\VF(\cdot)$ is the \Frechet variance defined in \eqref{equ:VF}, $\zeta(\kappa,c)=\sqrt{|\kappa|}c\coth(\sqrt{|\kappa|}c)$, and $\sigma_2(W)$ is the second largest singular value of $W$.
\end{lemma}

Lemma \ref{lma:4.1} imposes additional constraints compared to Lemma \ref{lma:3}, specifically lower sectional curvature bounds for $\cM$ and the boundedness of the points $\{y_{i,t+1}\}_{i=1}^n$. These assumptions ensure the proper behavior of the Riemannian gradient descent step in \eqref{equ:4.1} and enable the application of Lemma \ref{lma:2} in the analysis. Moreover, the contraction constant in Lemma \ref{lma:4.1},  $1-(1-\sigma_2(W))/(2C_3)$, is greater than the contraction constant in  Lemma \ref{lma:3}, which is $\sigma_2^2(W)$. This indicates that the above closed-form consensus mechanism requires more iterations to reach a given consensus level compared to the weighted \Frechet mean-based approach.

This trade-off underscores the balance between computational efficiency per round and the total number of rounds required for convergence. In an online learning context, where real-time processing and adaptability are crucial, the closed-form consensus mechanism is preferable as it improves the per-round computational efficiency. Therefore, in the following section, we integrate this closed-form consensus mechanism into the decentralized algorithm from Section \ref{sec:2.3} and demonstrate its effectiveness by establishing the same dynamic regret bounds as in Corollary \ref{corollary:1}.

\subsection{Improved Decentralized Algorithm and Theoretical Analysis}

Consider the same decentralized optimization setting as in Section \ref{sec:2}. Our improved decentralized projected Riemannian gradient descent (iDPRGD) modifies the original algorithm by replacing the consensus step \eqref{alg:II} with the closed-form consensus step \eqref{equ:4.1}. Specifically, we initialize $x_{i,1}=x_1\in\cX$ for all $i$. At time $t$, each agent receives a function $f_{i,t}$, and the following updates are executed:
\#
y_{i,t+1}&=\cP_{\cX}(\Exp_{x_{i,t}}(-\eta \nabla_{i,t})),\quad \nabla_{i,t}=\nabla f_{i,t}(x_{i,t}),\tag{III}\label{alg:III}\\
x_{i,t+1}&=\Exp_{y_{i,t+1}}\left(\gamma\sum_{j=1}^nw_{ij}\Log_{y_{i,t+1}}y_{j,t+1}\right), \tag{IV}\label{alg:IV}
\#
where $\cP_{\cX}$ is the projection operator onto $\cX$, $\eta$, $\gamma>0$ are the stepsizes, and $W=(w_{ij})$ is the weight matrix. By adopting the closed-form consensus step \eqref{alg:IV}, we avoid solving an additional minimization problem within the inner loop of the decentralized algorithm, significantly increasing computational speed, as  demonstrated in the numerical results section. 

In the remainder of this section, we establish the dynamic regret bound for this improved decentralized algorithm. Our analytical framework parallels that of Section \ref{sec:3}, with two key distinctions: first, we employ Lemma \ref{lma:4.1} to control the network error, and second, we use slightly different reasoning to bound the tracking error. The following lemmas summarize  these results, with their proofs deferred to the appendix.

\begin{lemma}[Network error II]\label{lma:4.2}
    Consider the problem setting in Section \ref{sec:2.2} and suppose the conditions in Section \ref{sec:2.3} hold. Let $\{x_{i,t}\}$ be the iterates  from the improved decentralized algorithm \eqref{alg:III} and \eqref{alg:IV} with $\gamma=1/(2C_3)$, where $C_3=\zeta(\kappa,2D)$, $\zeta(\kappa,c)$ is defined in Lemma \ref{lma:4.1}, and $D$ is the diameter of $\cX$. Then we have
    \$
    d(x_{i,t},\bar x_t)\leq \frac{4C_3\eta\sqrt{n}L}{1-\sigma_2(W)},
    \$
    where $\bar x_t$ is the \Frechet mean of $\{x_{i,t}\}_{i=1}^n$ and $\sigma_2(W)<1$ is the second largest singular value of $W$. 
\end{lemma}

The upper bound established in Lemma \ref{lma:4.2} exhibits a linear dependence on the stepsize $\eta$. By employing a diminishing stepsize sequence, we can ensure that the network error diminishes accordingly. This convergence property plays a crucial role in our dynamic regret analysis.

\begin{lemma}[Tracking error II]\label{lma:4.3}
    Consider the problem setting in Section \ref{sec:2.2} and assume the conditions in Section \ref{sec:2.4} hold. Let $\{x_{i,t},y_{i,t+1}\}$ be the iterates generated by the improved decentralized algorithm \eqref{alg:III} and \eqref{alg:IV} with $\gamma$ defined in Lemma \ref{lma:4.2}. Let $\{u_t\}$ be a sequence of points in the convex set $\cX$. Then it holds that
    \$
    \frac{1}{n}\sum_{i=1}^n\sum_{t=1}^T\left(d^2(u_t,x_{i,t})-d^2(u_t,y_{i,t+1})\right)\leq D^2+4DP_T,
    \$
    where $D$ is the diameter of $\cX$ and $P_T=\sum_{t=1}^Td(u_{t+1},u_t)$. 
\end{lemma}

Lemma \ref{lma:4.3} demonstrates that the tracking error depends on the nonstationarity measure $P_T$. Building on this, along with Lemma \ref{lma:4.2}, we can derive the dynamic regret bounds for the improved decentralized algorithm, as detailed in the following theorem. 

\begin{theorem}[Dynamic regret II]\label{thm:2}
    Consider the problem setting in Section \ref{sec:2.2} and assume the conditions in Section \ref{sec:2.4} hold. Let $\{x_{i,t}\}$ be the iterates generated by the improved decentralized algorithm \eqref{alg:III} and \eqref{alg:IV} with $\gamma$ given in Lemma \ref{lma:4.2} and some $\eta\leq 1$. Let $\{u_t\}$ be a sequence of points in the convex set $\cX$. Then we can bound the dynamic regret \eqref{equ:1.2}  as follows:
    \$
    \textnormal{D-Regret}(\{x_{i,t}\})\leq \eta C_4T+\frac{1}{2\eta}(D^2+4DP_T),
    \$
    where $C_4=(8C_3\sqrt{n}L^2)/(1-\sigma_2(W))+\zeta(\kappa,2D+L)L^2/2$, $C_3$, $\sigma_2(W)$, $D$, $\kappa$, and $\zeta(\kappa,c)$ are defined in Lemma \ref{lma:4.1}, $L$ is the Lipschitz constant, and $P_T$ is the path variation defined in Lemma \ref{lma:4.3}. 
\end{theorem}

The regret bound in Theorem \ref{thm:2} consists of two terms, $\eta C_4T$ and $\frac{1}{2\eta}(D^2+4DP_T)$. The first term depends linearly on $\eta$, while the second term depends inversely on $\eta$, illustrating the trade-off between consensus (small $\eta$) and dynamic adaptation (larger $\eta$). By selecting an optimal $\eta$, we can achieve a sublinear dynamic regret bound, as shown in the following corollary. This regret bound aligns with that of Corollary \ref{corollary:1}, demonstrating the effectiveness of the improved decentralized algorithm. 

\begin{corollary}
    Assume the same conditions as in Theorem \ref{thm:2} hold. Suppose the stepsize is givne by $\eta=\sqrt{D^2+4DP_T}/\sqrt{2C_4T}\leq 1$, where $C_4$ is defined in Theorem \ref{thm:1}. Then, we achieve a dynamic regret bound of order
    \$
    \cO\left(\sqrt{\frac{T(1+P_T)}{1-\sigma_2(W)}}\right).
    \$
\end{corollary}

\begin{figure}[!t]
    \centering
    \includegraphics[width=0.3\linewidth]{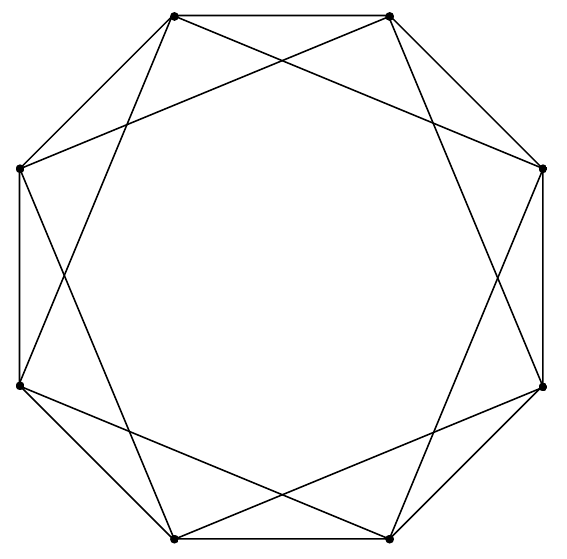}
    \caption{A network of $8$ agents. Each agent is connected to its 4 nearest neighbors.}
    \label{fig:1}
\end{figure}

\section{Numerical Studies}\label{sec:numerical}

In this section, we conduct extensive numerical experiments to demonstrate the effectiveness of the proposed decentralized algorithms. We investigate two Hadamard manifolds: hyperbolic spaces and the space of symmetric positive definite matrices. For each manifold, we focus on  the problem of decentralized \Frechet mean computation. Specifically, we consider a network with $n$ agents, where each agent receives a set of $K$ points $\{z_{i,k}^t\}_{k=1}^K$ at time $t$. The local objective function is the loss function for the \Frechet mean computation, defined as
\$
f_{i,t}(x)=\frac{1}{K}\sum_{k=1}^Kd^2(x,z_{i,k}^t),
\$
and the global function $f_t$ is given by \eqref{equ:global-function}. We assume $\{z_{i,k}^t\}\subseteq\cX$ in a bounded closed geodesically convex set. Therefore, the loss functions satisfy the conditions in Section \ref{sec:2}, such as geodesic convexity and the Lipschitz property over a bounded region. In the following section, we describe the nonstationary decentralized optimization settings in detail. We will examine both the decentralized algorithm in Section \ref{sec:2.3} and the improved decentralized algorithm in Section \ref{sec:4.1} in terms of dynamic regret with respect to the global minimizers and the computational efficiency.

\subsection{Hyperbolic Spaces}\label{sec:4.1}

Consider a two-dimensional hyperbolic space $\cM=\HH^2$ and a network of $n=40$ agents. Each agent is connected to its 4 nearest neighbors, as illustrated in Figure \ref{fig:1}. The weight matrix $W$ is constructed using the Metropolis constant edge weight method \cite{shi2015extra}. Specifically, $w_{ij}=1/5$ if agents $i$ and $j$ are connected or if $i=j$; otherwise $w_{ij}=0$. We will investigate the decentralized \Frechet mean computation task with $K=100$.

In our experiments, we examine the following decentralized online Riemannian optimization settings, representing various nonstationary scenarios. For hyperbolic space data, we use the hyperboloid model for experiments and the \Poincare disk model for visualization. Details of these models are provided in Appendix \ref{sec:a6}.

{\it (A)  Static.} In our first setting, we assume that $\{f_{i,t}\}$ remain constant over time $t$. Each agent draws data points $\{z_{i,k}\}_{k=1}^K$ from  Riemannian radial distributions \cite{chen2024riemannian}. Specifically, we generate $n$ samples $\{z_i\}_{i=1}^n$ from a Riemannian Gaussian distribution with base point $\alpha=(1,0,0)$ in the hyperboloid model and $\sigma=5$. These $\{z_i\}_{i=1}^n$ serve as the base points for each agent, and $\{z_{ik}\}_{k=1}^K$ are generated from the Riemannian Gaussian distribution with base point $\alpha=z_i$ and $\sigma=1$. The initial state $x_{i,1}$ of each agent is set as the \Frechet mean of the local data points $\{z_{i,k}\}_{k=1}^K$. We then implement both DPRGD and iDPRGD with stepsizes $\eta =0.001$ and $\gamma=1$, leading to the local states $\{x_{i,t}\}$. Figure \ref{fig:2} shows these local states along with the global minimizer in the \Poincare disk. The results demonstrate that, for both DPRGD and iDPRGD, all local states converge to the global minimizer, showcasing the effectiveness of the algorithms. 

\begin{figure}
    \centering
    \begin{subfigure}[b]{0.19\linewidth}\includegraphics[width=\linewidth]{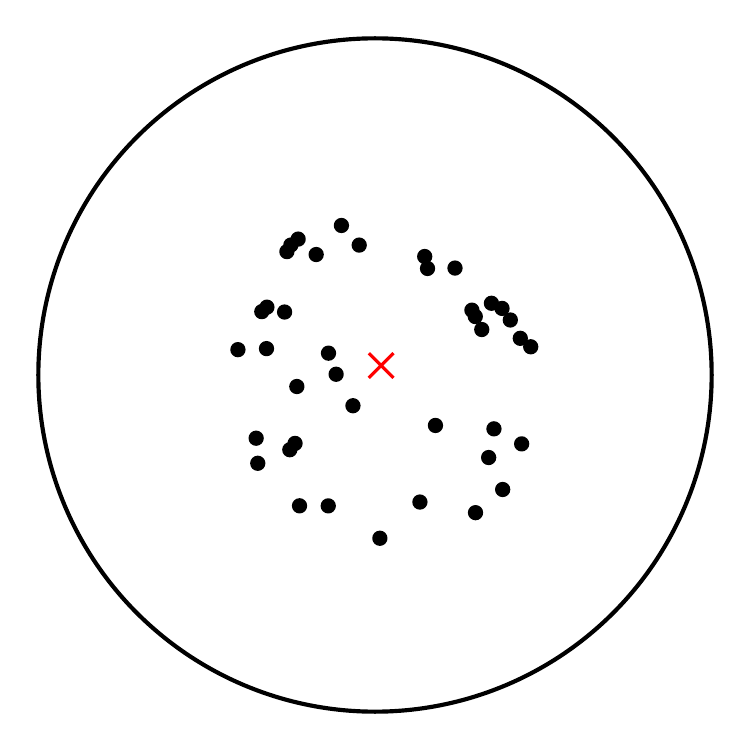}
    \caption{$t=1$}
    \end{subfigure} 
    \begin{subfigure}[b]{0.19\linewidth}\includegraphics[width=\textwidth]{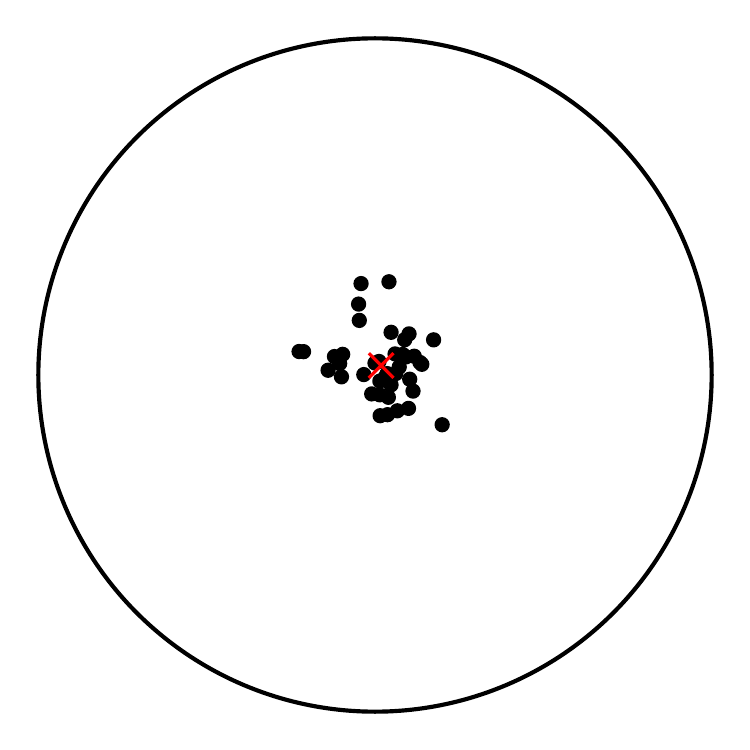}
    \caption{$t=2$}
    \end{subfigure} 
    \begin{subfigure}[b]{0.19\linewidth}
    \includegraphics[width=\textwidth]{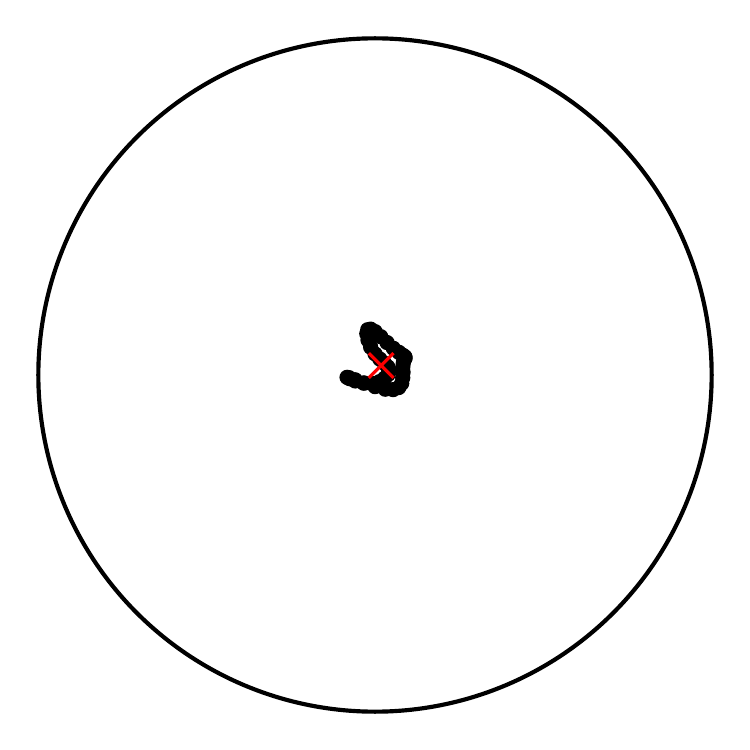}
    \caption{$t=6$}
    \end{subfigure}
    \begin{subfigure}[b]{0.19\linewidth}
        \includegraphics[width=\textwidth]{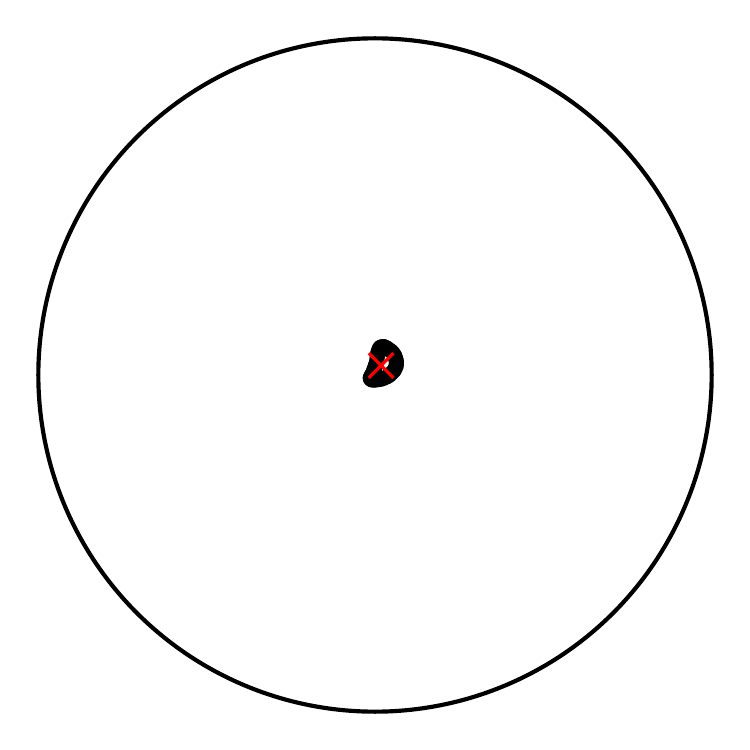}
        \caption{$t=21$}
    \end{subfigure}
    \begin{subfigure}[b]{0.19\linewidth}
        \includegraphics[width=\textwidth]{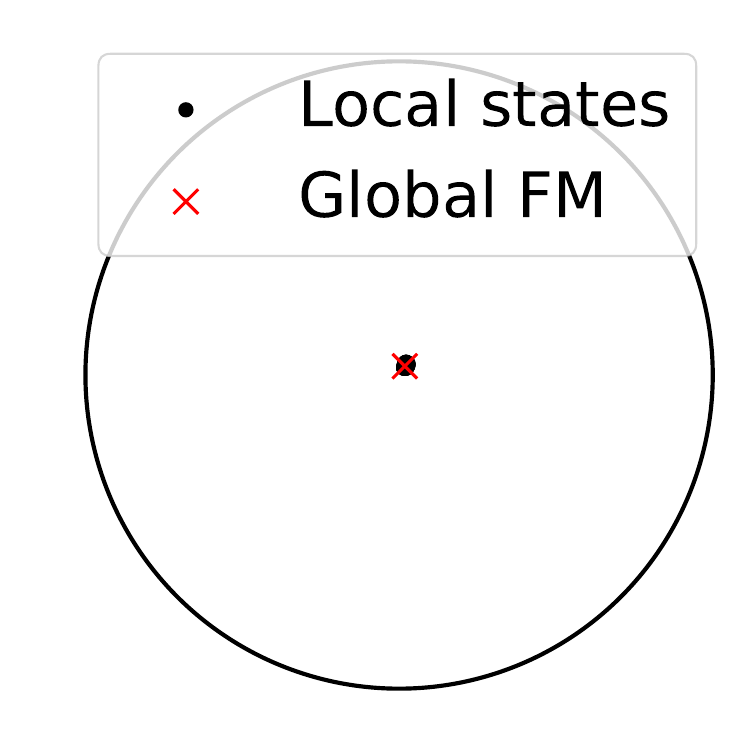}
        \caption{$t=101$}
    \end{subfigure}
    
    \begin{subfigure}[b]{0.19\linewidth}
        \includegraphics[width=\textwidth]{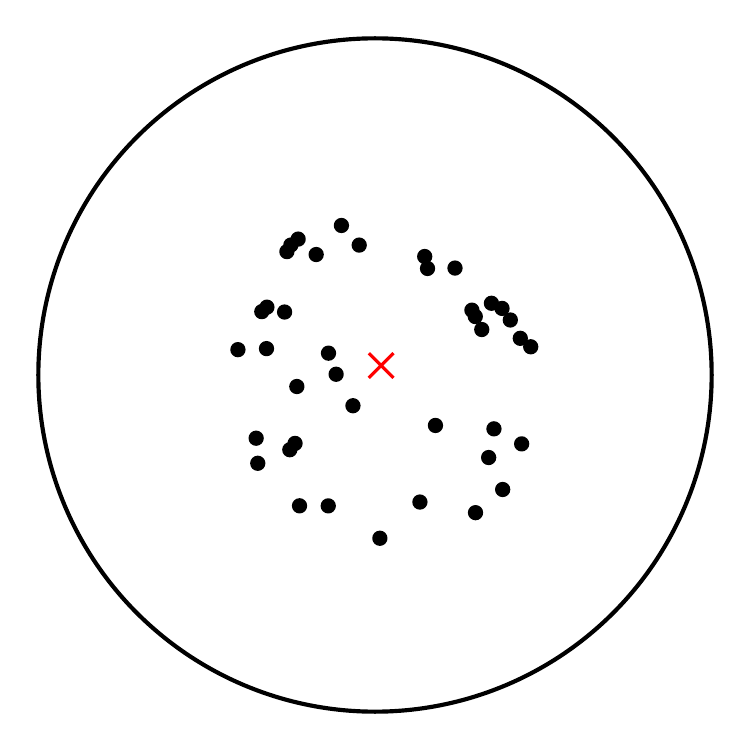}
        \caption{$t=1$}
    \end{subfigure}
    \begin{subfigure}[b]{0.19\linewidth}
        \includegraphics[width=\textwidth]{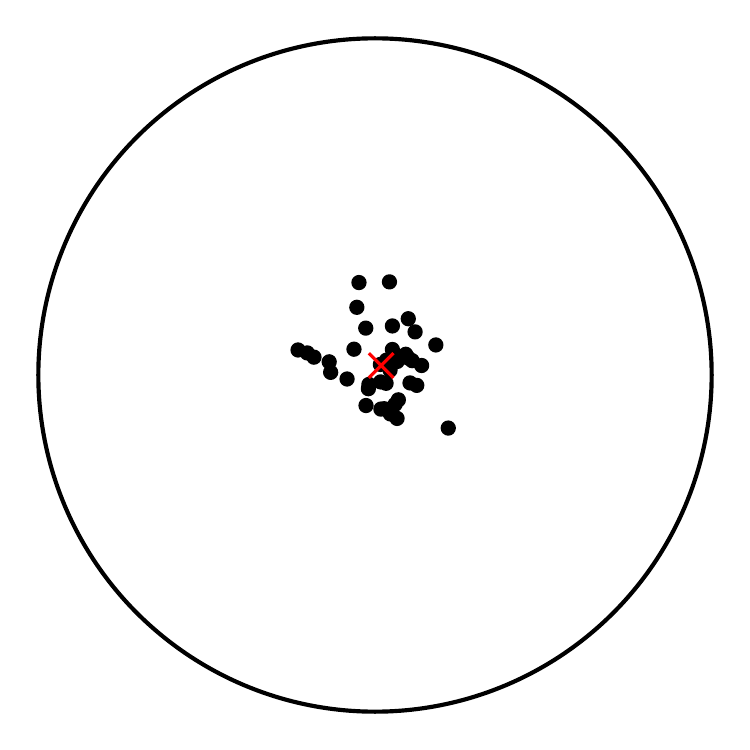}
        \caption{$t=2$}
    \end{subfigure}
    \begin{subfigure}[b]{0.19\linewidth}
        \includegraphics[width=\textwidth]{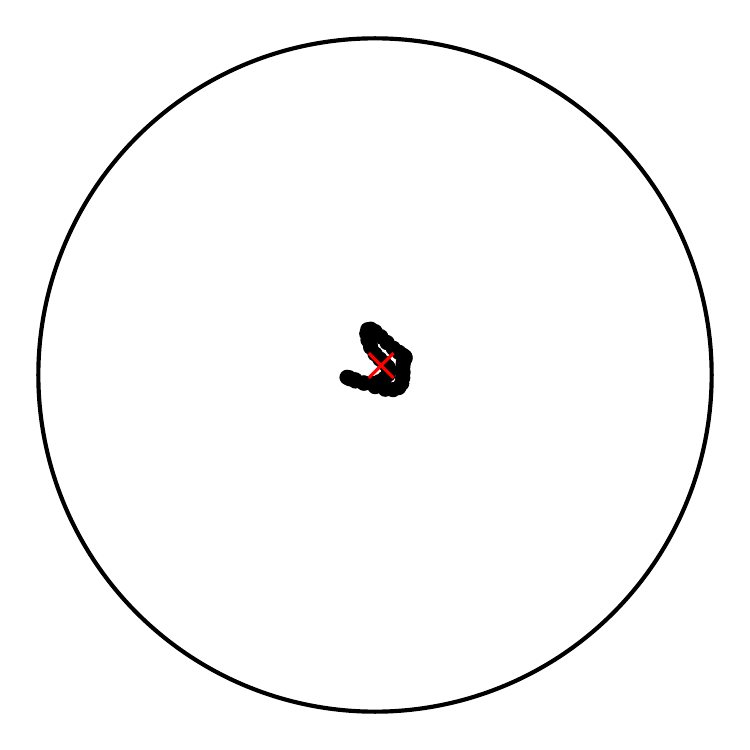}
        \caption{$t=6$}
    \end{subfigure}
    \begin{subfigure}[b]{0.19\linewidth}
        \includegraphics[width=\textwidth]{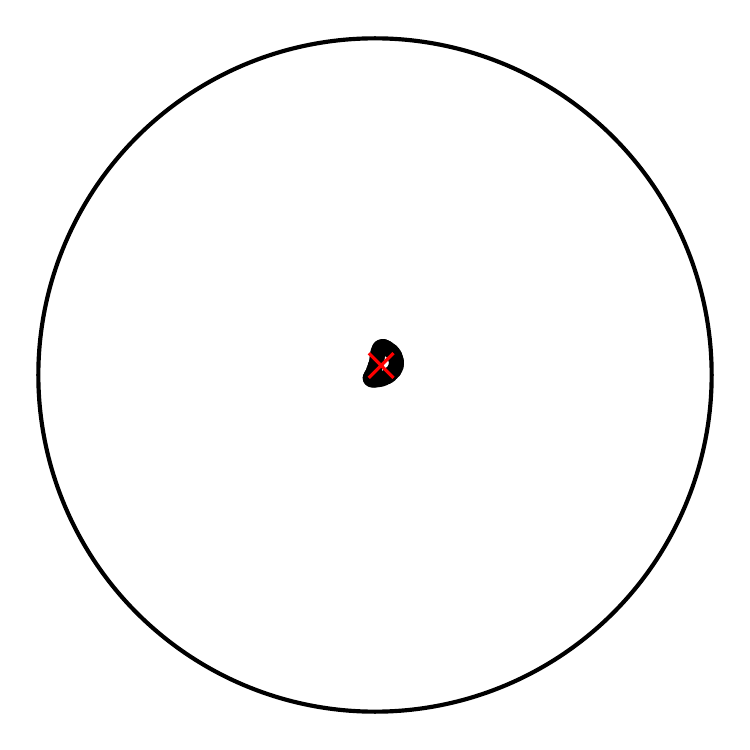}
        \caption{$t=21$}
    \end{subfigure}
    \begin{subfigure}[b]{0.19\linewidth}
        \includegraphics[width=\textwidth]{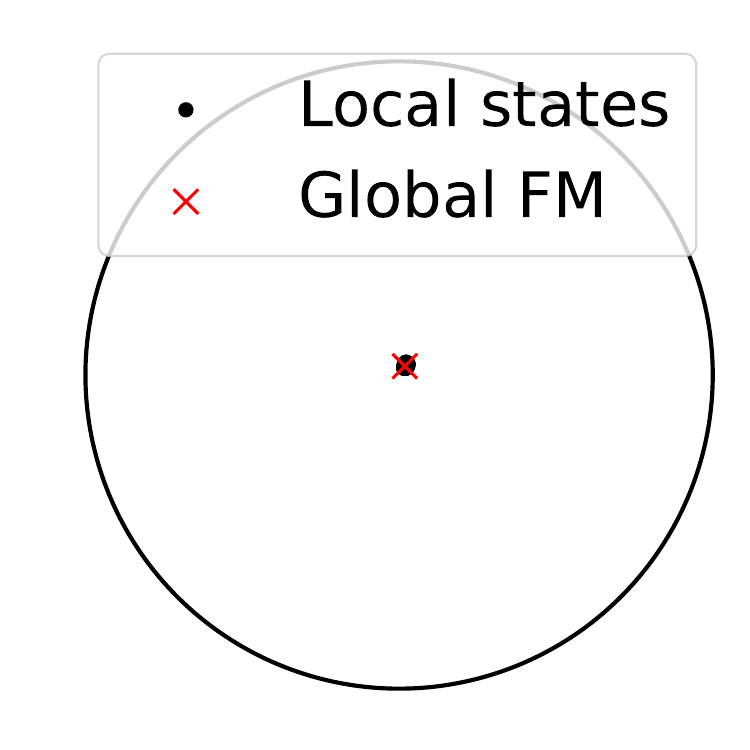}
        \caption{$t=101$}
    \end{subfigure}
    \caption{Illustration of the experiment (A) for hyperbolic spaces. The first row shows the results for DPRGD and the second row shows the results for iDPRGD. In each panel, we display the local states of $n$ agents and the global minimizer (\Frechet mean) at time $t=1,2,6,21,101$ in the \Poincare disk.}
    \label{fig:2}
\end{figure}

\begin{figure}
    \centering
    \includegraphics[width=0.45\linewidth]{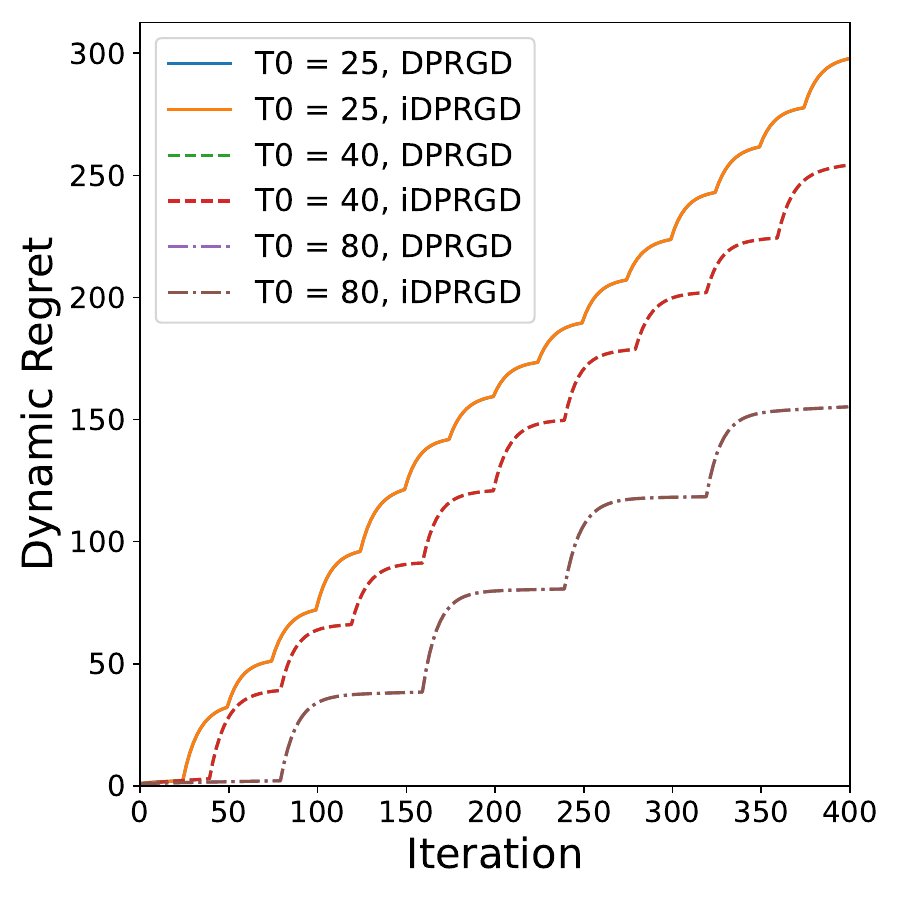}
    \includegraphics[width=0.45\linewidth]{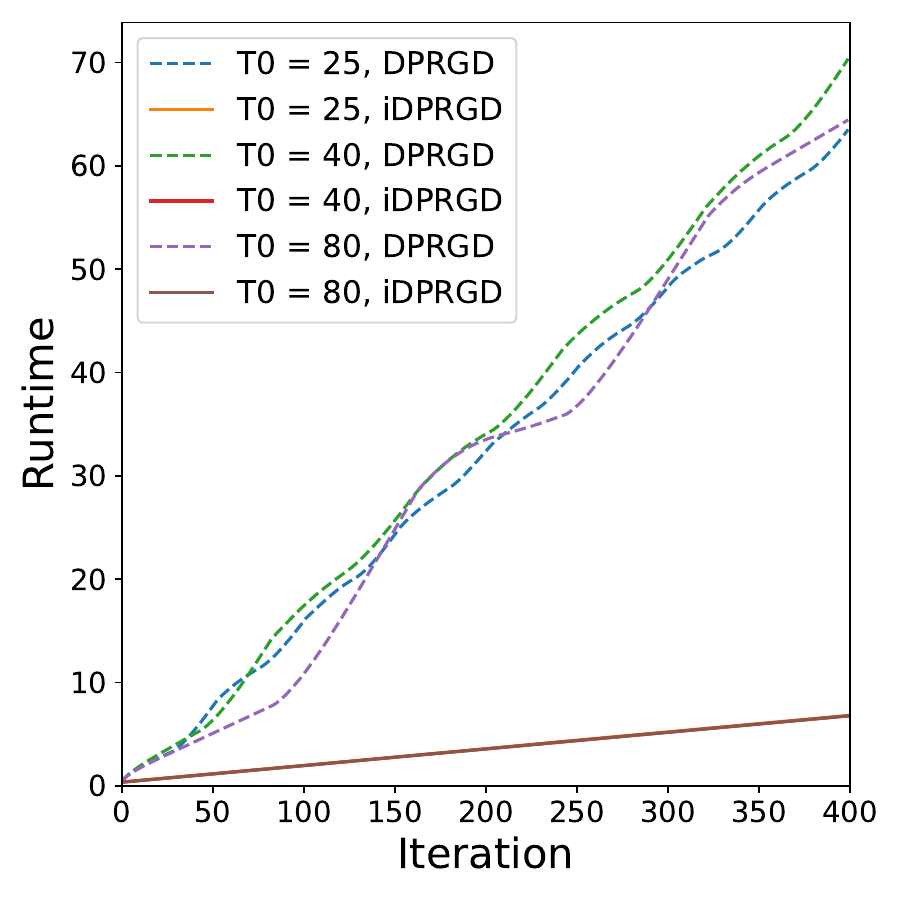}
    \caption{The left and right panels visualize the dynamic regret and runtime of DPRGD and iDPRGD in the experiment (B) for hyperbolic spaces, respectively. The runtime is measured in seconds.}
    \label{fig:3}
\end{figure}

    {\it (B) Abrupt changing.} Our second setting assumes that the local functions $\{f_{i,t}\}$ change abruptly. Specifically, we consider $T=400$ and let $\{f_{i,t}\}$ change every $T_0$ steps. When $t/T_0$ is an odd integer, we draw $\{z_i^t\}_{i=1}^n$ from a Riemannian Gaussian distribution with base point $\alpha=(3,2,2)$ and $\sigma=5$, followed by drawing samples $\{z_{ik}^t\}_{k=1}^K$ from a Riemannian Gaussian distribution with base point $\alpha=z_{i}^t$ and $\sigma=1$. These $\{z_{ik}^t\}$ remain unchanged until $t/T_0$ reaches the next integer. When $t/T_0$ is an even integer, we generate $\{z_i^t\}_{i=1}^n$ from a Riemannian Gaussian distribution with base point $\alpha=(1,0,0)$ and $\sigma=5$, and draw $\{z_{ik}^t\}_{k=1}^K$ from a Riemannian Gaussian distribution with $\alpha=z_i^t$ and $\sigma=1$. These samples, $\{z_{ik}^t\}$, remain unchanged until the next integer value of $t/T_0$. When $T_0$ is larger, the functions change less frequently, creating a more stationary environment. To quantify this, we examine $T_0\in\{25, 40, 80\}$ in our experiments.

    We apply both DPRGD and iDPRGD to these data using $\eta=0.05$ and $\gamma=1$. Figure \ref{fig:3} shows the dynamic regret and runtime for these two algorithms across three dynamic settings. The results indicate that more stationary environments ($T_0 = 80$) yield lower dynamic regret. Sharp increases in regret occur when distributions change, contrasting with plateaus during static periods. This highlights the algorithm's ability to swiftly adapt to new distributions while maintaining stability in unchanging environments. Moreover, while the dynamic regret performance of DPRGD and iDPRGD is nearly identical, iDPRGD achieves a significantly reduced runtime due to the lower computational burden in the consensus step.

    {\it (C)  General changing.} In our final setting, we assume that local functions $\{f_{i,t}\}$ change both abruptly and gradually. Specifically, we let $T=400$ and $T_0\in\{25,40,80\}$. For any time $t\leq T$, we rewrite it as $t=sT_0+r$, where $s$ and $r$ are nonnegative integers. We then draw $\{z_i^t\}$ from a Riemannian Gaussian distribution with base point
    \$
    \alpha=\cosh(r/T_0)(1,0,0)+\sinh(r/T_0)(0,1/\sqrt{2},1/\sqrt{2}),
    \$
    and $\sigma=5$. Subsequently, we draw samples $\{z_{ik}^t\}_{k=1}^K$ from a Riemannian Gaussian distribution with $\alpha=z_{i}^t$ and $\sigma=1$.

    After generating the data, we implement DPRGD and iDPRGD with stepsizes $\eta=0.05$ and $\gamma=1$. The dynamic regrets and runtime across three dynamic settings are visualized in Figure \ref{fig:4}. The results show that the more stationary environment $(T_0=80)$ yields lower dynamic regrets. Sharp increases in regret when distributions change abruptly at time $sT_0$ for integers $s$. In contrast, when distributions change gradually, the regret increases at a slower rate, indicating the decentralized method's ability to adapt to a new environment even during gradual changes. Moreover, while the dynamic regret performance of DPRGD and iDPRGD is nearly identical, iDPRGD demonstrates significantly lower computational cost, highlighting the effectiveness of the closed-form consensus step.

\subsection{Symmetric Positive Definite Matrices}\label{sec:4.2}

Let $\cM$ denote the space of $3\times 3$ symmetric positive definite matrices. The objective of this section is to examine the decentralized \Frechet mean problem on $\cM$. In our setup, we set $K=20$ and use a network of $n=10$ agents, similar to the one in Figure \ref{fig:1}. We explore the following dynamic environment. 

{\it Setting.} Let $T=80$. Assume that local functions $\{f_{i,t}\}$ change abruptly during the first 40 steps and change both abruptly and gradually in the remaining steps. Specifically, for $t\leq 40$, we let $\{f_{i,t}\}$ change every $T_0$ steps, where $T_0\in\{5,10,20\}$ represents different levels of nonstationarity. When $t/T_0$ is an even integer, we draw $Z_{i}^t\sim\Exp_I(V)$,  where $I$ is the identity matrix and $V$ is a random symmetric matrix with entries sampled from $\cU(0,0.1)$. We then generate $Z_{ik}^t\sim \Exp_{Z_i^t}(V)$, where $V$ is a symmetric matrix with entries following $\cU(0,0.1)$. The set $\{Z_{ik}^t\}$ remains unchanged until $t/T_0$ reaches the next integer. When $t/T_0$ is an odd integer, we repeat the same process but with $I$ replaced by $3I$. For $t>40$, we express $t$ as $t=sT_0+r$, where   $s$ and $r$ are nonnegative integers. At time $t$, we draw $\{Z_i^t\}\sim \Exp_{(1+2r/T_0)I}(V)$,  where $V$ is a symmetric matrix with entries generated from $\cU(0,0.1)$. Then we generate $\{Z_{ik}^t\}\sim\Exp_{Z_i^t}(V)$,  where $V$ is a symmetric matrix with entries following $\cU(0,0.1)$.


{\it Results.}
After generating the data, we implement both DPRGD and iDPRGD with stepsizes $\eta=0.1$ and $\gamma=0.5$. The dynamic regrets and runtime for the three dynamic settings are visualized in Figure \ref{fig:5}. The results show that more static cases $(T_0=20)$ yield lower dynamic regrets. Sharp increases in regrets are observed when functions change abruptly, in contrast to plateaus when functions are the same or change gradually. There regret trends demonstrate  the decentralized algorithm's ability to swiftly adapt to new environments,  even when changes are gradual. Additionally, the right panel in Figure \ref{fig:5} highlights the computational efficiency of iDPRGD compared to DPRGD. 
\begin{figure}
    \centering
    \includegraphics[width=0.45\linewidth]{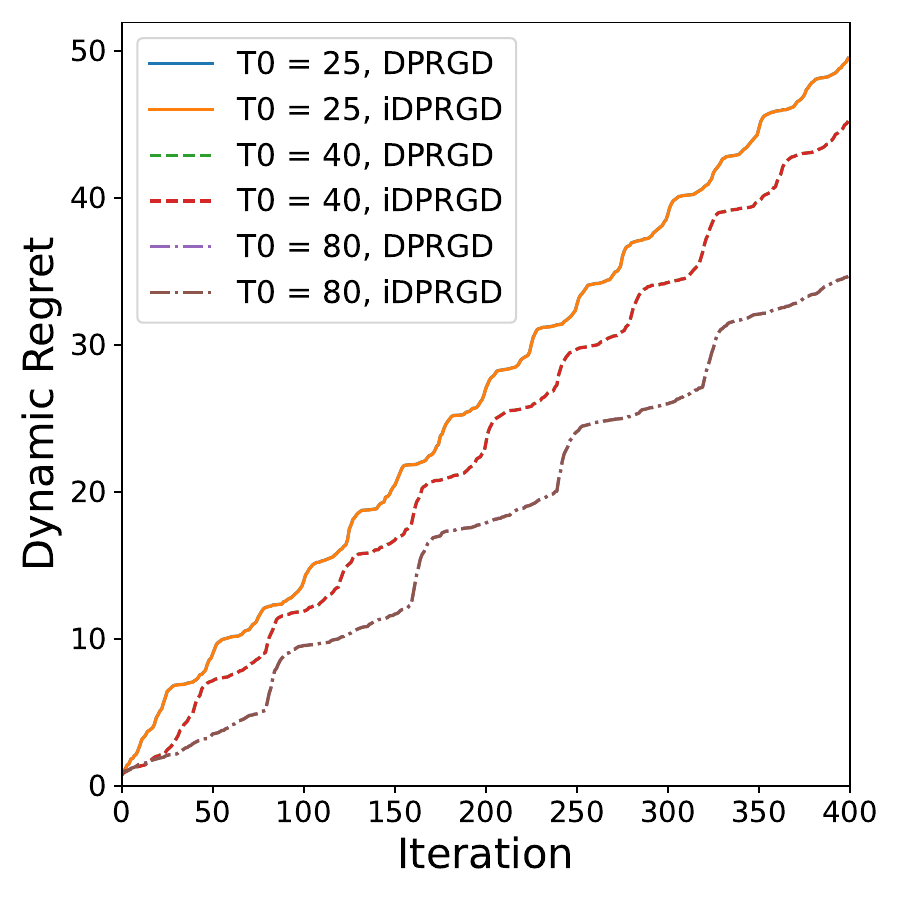}
    \includegraphics[width=0.45\linewidth]{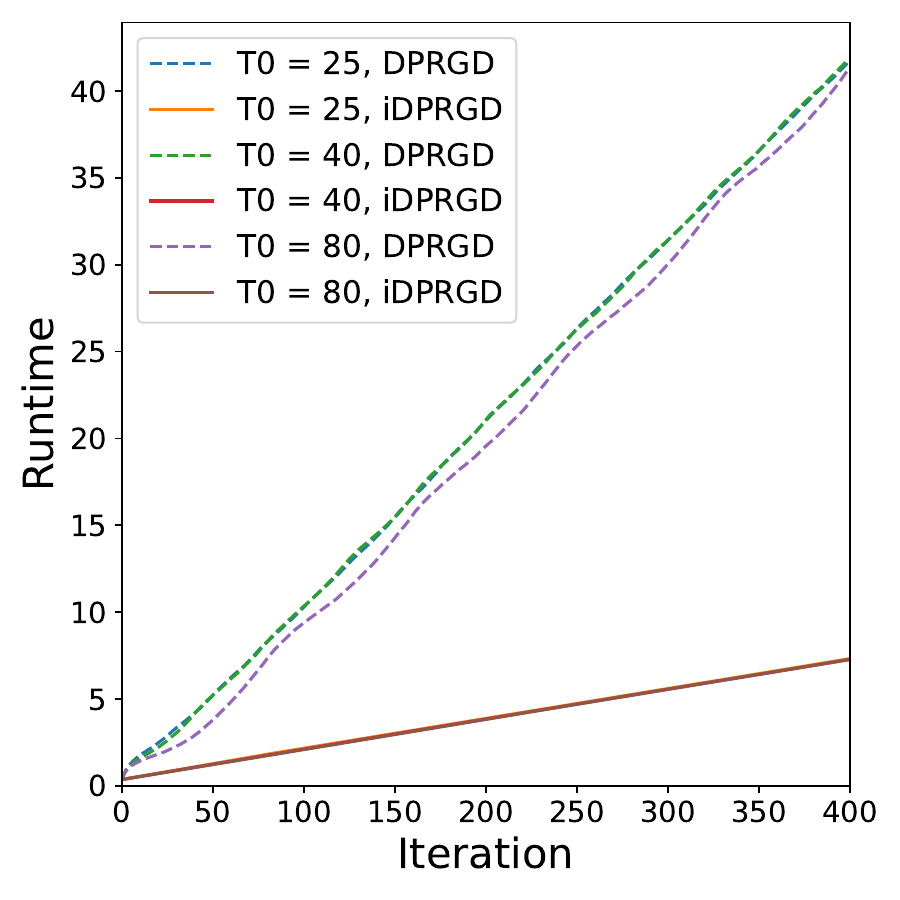}
    \caption{The left and right panels visualize the dynamic regret and runtime of DPRGD and iDPRGD in the experiment (C) for hyperbolic spaces, respectively. The runtime is measured in seconds.}
    \label{fig:4}
\end{figure}


\section{Conclusion}\label{sec:conclusion}


In this paper, we propose the first decentralized online geodesically convex optimization algorithm. We demonstrate its effectiveness by deriving a dynamic regret bound, highlighting its adaptability to dynamic environments. When restricted to the Euclidean, static, and centralized cases, our results recover the corresponding theories for these settings. A caveat in our original algorithm is that the consensus step involves exactly solving a minimization problem, which can be computationally expensive. To further enhance computational efficiency, we propose a computationally efficient consensus step and incorporate it into the decentralized optimization algorithm. This modification reduces computational cost while maintaining the same dynamic regret bound. Finally, we conclude by posing the following  questions for future research. 
\begin{enumerate}
\item {\it Adaptivity.} In practice, it is crucial to design an adaptive algorithm without using unknown oracle information. For example, it remains an important open question on tuning the stepsize, even for Euclidean convex optimization problems \cite{shahrampour2017distributed}. 

\item \textit{Extensions.} It would be valuable to extend this work to address time-varying directed graphs, constrained optimization, bandit optimization, accelerated methods, or optimization with a strongly convex and smooth global function.
\end{enumerate}

\begin{figure}
    \centering
    \includegraphics[width=0.45\linewidth]{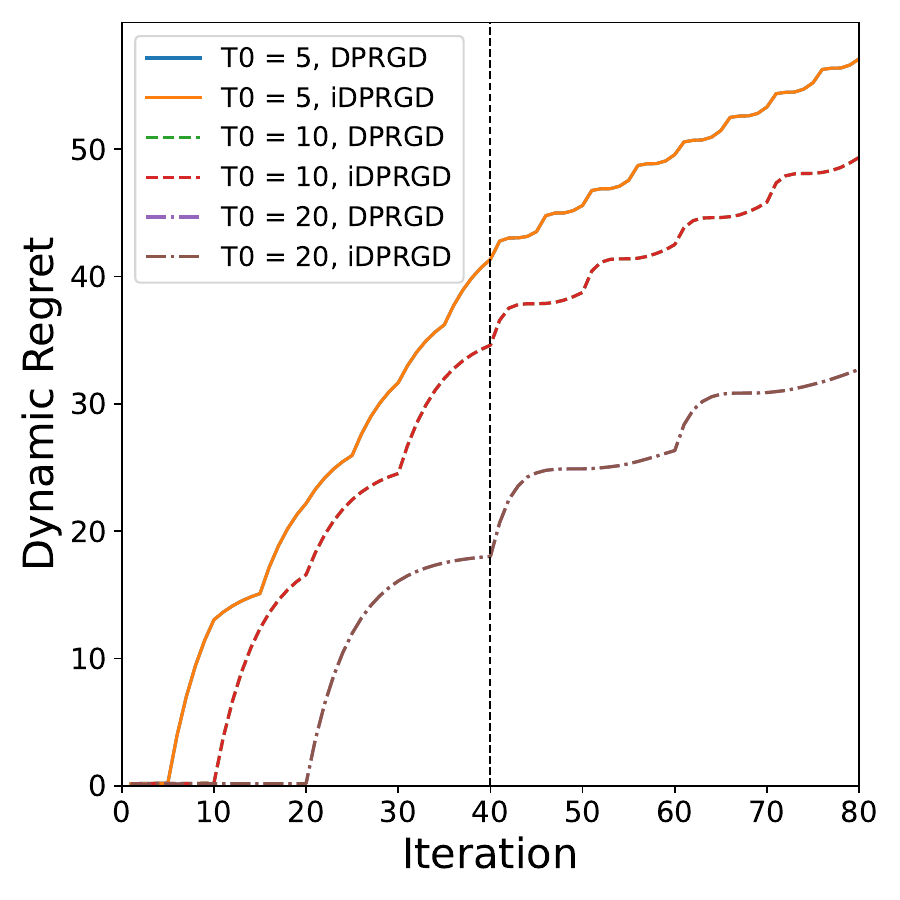}
    \includegraphics[width=0.45\linewidth]{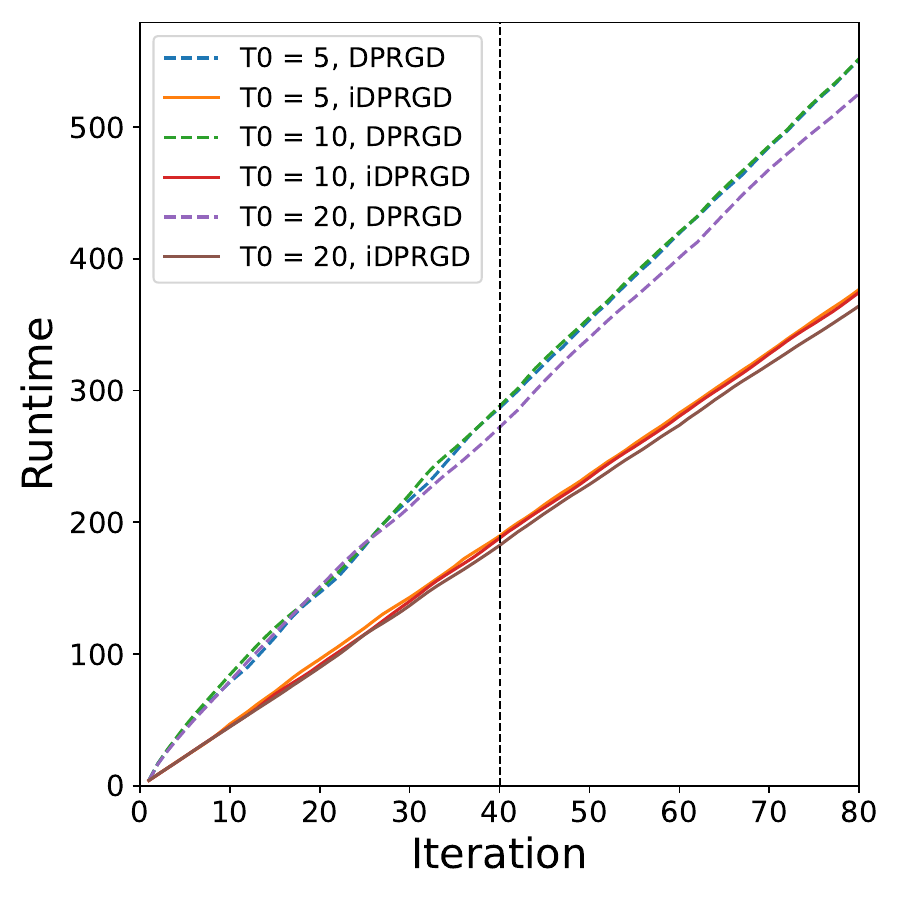}
    \caption{The left and right panels visualize the dynamic regret and runtime of DPRGD and iDPRGD in the experiment in Section \ref{sec:4.2}, respectively.   Periods $t\leq 40$ and $t>40$ represent different dynamic scenarios. The runtime is measured in seconds.}
    \label{fig:5}
\end{figure}






\appendix

\section{Geometric lemmas}

In this section, we present geometric lemmas used in our proof. Lemma \ref{lma:a1} is concerned with geodesic triangles in a Hadamard manifold. It is complementary to Lemma \ref{lma:2} by using the nonpositive sectional curvatures of a Hadamard manifold.

\begin{lemma}[\cite{sakai1996riemannian}, Proposition 4.5]\label{lma:a1}
Suppose $\cM$ is a Hadamard manifold. For any geodesic triangle in $\cM$ with side lengths $a,b,c$, we have
\$
a^2\geq b^2+c^2-2bc\cos A,
\$
where $A$ is the opposite angle associated with $a$. 
\end{lemma}

Lemma \ref{lma:a2} provides the Jensen inequality on a Hadamard manifold. 

\begin{lemma}[\cite{sturm2003probability}, Theorem 6.2]\label{lma:a2}
    Let $\{y_i\}_{i=1}^n$ be $n$ points on a Hadamard manifold $\cM$. Let $y^w=\argmin_{y}\sum_iw_id^2(y,y_i)$ be the weighted \Frechet mean, where $w_i\geq 0$ and $\sum_{i=1}^nw_i=1$. Then for any geodesically convex function $f$ on $\cM$, it holds that
    \$
    f(y^w)\leq \sum_iw_if(y_i). 
    \$
\end{lemma}

\section{Proof of Lemma \ref{lma:3}}
    Recall that the consensus step \eqref{alg:II} maps $n$ points $\{y_{i,t+1}\}$ to new points $\{x_{i,t+1}\}$.
    Let $\bar y_{t+1}$ and $\bar x_{t+1}$ denote the \Frechet means of $\{y_{i,t+1}\}$ and $\{x_{i,t+1}\}$, respectively:
    \#
    \bar y_{t+1}&=\argmin_y\sum_{i}d^2(y,y_{i,t+1}),\label{equ:ybar-t}\\
    \bar x_{t+1}&=\argmin_x\sum_id^2(x,x_{i,t+1}).\label{equ:ybar-t+1}
    \#
    Then applying Theorem 2.4 in \cite{bacak2014computing} to weighted \Frechet means in \eqref{equ:2.1}, we obtain that
    \$
    d^2(\bar y_{t+1},x_{i,t+1})+\sum_jw_{ij}d^2(x_{i,t+1},y_{j,t+1})\leq \sum_{j}w_{ij}d^2(\bar y_{t+1},y_{j,t+1})
    \$
    for all $i=1,\ldots,n$. Summing up these $n$ inequalities, we have  
    \#
    &\sum_{i}d^2(\bar y_{t+1},x_{i,t+1})+\sum_i\sum_jw_{ij}d^2(x_{i,t+1},y_{j,t+1})\notag\\
    \leq &\sum_i\sum_jw_{ij}d^2(\bar y_{t+1},y_{j,t+1})\notag\\
    =&\sum_jd^2(\bar y_{t+1},y_{j,t+1}),\label{equ:3.2}
    \#
    where the second inequality uses the fact that $\sum_iw_{ij}=1$. Applying Lemma \ref{lma:a1} to the geodesic triangle with vertices   $\bar y_{t+1},x_{i,t+1},y_{j,t+1}$, we have
    \$
    &d^2(x_{i,t+1},y_{j,t+1}) \\
    \geq\ &d^2(\bar y_{t+1},x_{i,t+1})+d^2(\bar y_{t+1},y_{j,t+1})-2\inner{\Log_{\bar y_{t+1}}x_{i,t+1}}{\Log_{\bar y_{t+1}}y_{j,t+1}},
    \$
    where $\Log_{\bar y_{t+1}}(\cdot)$ is the Riemannian logarithm map and $\inner{\cdot}{\cdot}$ is the inner product in the tangent space $T_{\bar y_{t+1}}\cM$. Substituting this into \eqref{equ:3.2}, we obtain that 
    \$
    &\sum_id^2(\bar y_{t+1},x_{i,t+1})+\sum_{ij}w_{ij}\left(d^2(\bar y_{t+1},x_{i,t+1})+d^2(\bar y_{t+1},y_{j,t+1})\right)\\
    \leq\ &2\sum_{ij}w_{ij} \inner{\Log_{\bar y_{t+1}}x_{i,t+1}}{\Log_{\bar y_{t+1}}y_{j,t+1}}+\sum_jd^2(\bar y_{t+1},y_{j,t+1}).
    \$
    Since $\sum_iw_{ij}=\sum_jw_{ij}=1$, this inequality can be simplified as follows:
    \#\label{equ:3.3}
    \sum_id^2(\bar y_{t+1},x_{i,t+1})\leq \sum_{ij}w_{ij} \inner{\Log_{\bar y_{t+1}}x_{i,t+1}}{\Log_{\bar y_{t+1}}y_{j,t+1}}. 
    \#
    Since $\bar y_{t+1}$ is the \Frechet mean of $\{y_{i,t+1}\}$, it holds that $
    \sum_j\Log_{\bar y_{t+1}}y_{j,t+1}=0.$
    Therefore, we can rewrite \eqref{equ:3.3} as follows:
    \$
    \sum_id^2(\bar y_{t+1},x_{i,t+1})\leq \sum_{ij}(w_{ij}-\frac{1}{n}) \inner{\Log_{\bar y_{t+1}}x_{i,t+1}}{\Log_{\bar y_{t+1}}y_{j,t+1}}.
    \$
    The right hand side of the inequality equals to the Frobenius inner product between $W-\frac{1}{n}$ and the matrix 
    \$
    Q=\left(\inner{\Log_{\bar y_{t+1}}x_{i,t+1}}{\Log_{\bar y_{t+1}}y_{j,t+1}}\right)\in\RR^{n\times n},
    \$
    where $W=(w_{ij})$. By selecting an appropriate orthonormal basis in the tangent space $T_{\bar y_{t+1}}\cM$, we can regard $\{\Log_{\bar y_{t+1}}x_{i,t+1}\}$ and $\{\Log_{\bar y_{t+1}}y_{j,t+1}\}$ as vectors in a Euclidean space, and write
    \$
    V_{t+1}&=\left(\Log_{\bar y_{t+1}}x_{1,t+1},\ldots,\Log_{\bar y_{t+1}}x_{n,t+1}\right),\\
    V_t&=\left(\Log_{\bar y_{t+1}}y_{1,t+1},\ldots, \Log_{\bar y_{t+1}}y_{n,t+1}\right).
    \$
    Then the matrix $Q$ can be rewritten as the matrix multiplication $Q=V_{t+1}^\top V_{t}$. Therefore,
    \$
    \sum_{ij}(w_{ij}-\frac{1}{n}) \inner{\Log_{\bar y_{t+1}}x_{i,t+1}}{\Log_{\bar y_{t+1}}y_{j,t+1}} 
    = \tr\left(\left(W-\frac{1}{n}\right)V_{t+1}^\top V_t\right),
    \$
    where we use the fact that $W$ is a symmetric matrix. By employing the singular value inequality and  Cauchy-Schwarz inequality, we have
    \$
    \tr\left(\left(W-\frac{1}{n}\right)V_{t+1}^\top V_t\right)&\leq \sigma_1\left(W-\frac{1}{n}\right)\tr(V^\top_{t+1}V_t) \leq \sigma_1\left(W-\frac{1}{n}\right)\norm{V_{t+1}}_{\rF}\norm{V_t}_{\rF},
    \$
    where $\norm{\cdot}_{\rF}$ denotes the Frobenius norm. Since $W$ is a doubly stochastic matrix, we have $
    \sigma_1(W-\frac{1}{n})=\sigma_2(W).$
    Also, we observe that 
    \$
    \norm{V_{t+1}}_{\rF}^2&=\sum_id^2(\bar y_{t+1},x_{i,t+1}),\\
    \norm{V_t}_{\rF}^2&=\sum_jd^2(\bar y_{t+1},y_{j,t+1}).
    \$
    Therefore, combining all the above results, we obtain that
    \$
    \norm{V_{t+1}}_{\rF}^2&\leq \tr\left(\left(W-\frac{1}{n}\right)V_{t+1}^\top V_t\right) \leq \sigma_2(W)\norm{V_{t+1}}_{\rF}\norm{V_t}_{\rF},
    \$
    which implies that $\norm{V_{t+1}}_{\rF}\leq \sigma_2(W)\norm{V_t}_{\rF}$. Finally, using the definition of $\bar x_{t+1}$ and \Frechet variance, we have
    \$
    n\VF(\{x_{i,t+1}\})&=\sum_id^2(\bar x_{t+1},x_{i,t+1})\\
    &\leq \sum_{i}d^2(\bar y_{t+1},x_{i,t+1})\\
    &\leq \sigma_2^2(W)\sum_jd^2(\bar y_{t+1},y_{j,t+1})=n\sigma_2^2(W)\VF(\{y_{i,t+1}\}).
    \$
    This concludes the proof.  

\section{Proof of Lemma \ref{lma:4}}

 We prove this lemma by investigating the \Frechet variance $\VF(\{x_{i,t}\})$ of the set $\{x_{i,t}\}_{i=1}^n$ over time $t$. As we initialize $x_{i,1}=x_1$ at the same point, we have $\VF(\{x_{i,1}\})=0.$ At each iteration, we first perform the projected Riemannian gradient descent \eqref{alg:I}, so we have 
 \#\label{equ:b1}
 d(y_{i,t+1},x_{i,t})\leq \eta \norm{\nabla_{i,t}}\leq \eta L,
 \#
 where the first inequality follows the argument in Section \ref{sec:2.4} and the second inequality uses the fact that $f_{i,t}$ is $L$-Lipschitz. Therefore, using the triangle inequality, we have that
 \$
 \sqrt{\sum_id^2(\bar x_t,y_{i,t+1})} 
 \leq & \sqrt{\sum_id^2(\bar x_t,x_{i,t})}+\sqrt{\sum_id^2(y_{i,t+1},x_{i,t})}\\
 \leq &\sqrt{\sum_id^2(\bar x_t,x_{i,t})}+\eta\sqrt{n}L, 
 \$
 where $\bar x_t$ is the \Frechet mean of $\{x_{i,t}\}_{i=1}^n$ and the second inequality uses \eqref{equ:b1}. Let $\bar y_{t+1}$ be the \Frechet mean of $\{y_{i,t+1}\}_{i=1}^n$. Then we have
 \#
 \sqrt{\sum_id^2(\bar y_{t+1},y_{i,t+1})}&\leq \sqrt{\sum_id^2(\bar x_t,y_{i,t+1})} 
\leq \sqrt{\sum_id^2(\bar x_t,x_{i,t})}+\eta\sqrt{n}L.\label{equ:b2}
 \#
 Using Lemma \ref{lma:3}, we have
 \$
 \sqrt{\sum_id^2(\bar x_{t+1},x_{i,t+1})} 
 \leq\ &\sigma_2(W) \sqrt{\sum_id^2(\bar y_{t+1},y_{i,t+1})}\\
 \leq\ &\sigma_2(W)\sqrt{\sum_id^2(\bar x_t,x_{i,t})}+\eta\sigma_2(W)\sqrt{n}L,
 \$
 where the second inequality uses \eqref{equ:b2}. Then using an induction argument, one can easily show that 
 \$
 \sqrt{\sum_id^2(\bar x_{t},x_{i,t})}\leq \frac{\eta\sigma_2(W)\sqrt{n}L}{1-\sigma_2(W)},\quad\forall t.
 \$
This immediately implies that
\$
d(x_{i,t},\bar x_t)\leq  \sqrt{\sum_id^2(\bar x_{t},x_{i,t})}\leq \frac{\eta\sigma_2(W)\sqrt{n}L}{1-\sigma_2(W)},\quad \forall i,t.
\$ 

\section{Proof of Lemma \ref{lma:5}}
First, we adopt the following decomposition:
\$
d^2(u_t,x_{i,t})-d^2(u_t,y_{i,t+1})=\ &d^2(u_t,x_{i,t})-d^2(u_{t+1},x_{i,t+1})\\
+\ &d^2(u_{t+1},x_{i,t+1})-d^2(u_{t},x_{i,t+1})\\
+\ &d^2(u_{t},x_{i,t+1})-d^2(u_t,y_{i,t+1}).
\$
Then we upper bound the above three terms separately. For the second term, we use the triangle inequality to obtain that
\$
d^2(u_{t+1},x_{i,t+1})-d^2(u_{t},x_{i,t+1})\leq 2Dd(u_{t+1},u_{t}) 
\$
where $D$ is the diameter of $\cX$ and we use the fact that $u_{t+1},u_t,x_{i,t+1}\subseteq\cX$. For the third term, we first observe that $\rho(z)\coloneqq d^2(u_t,z)$ is a geodesically convex function over a Hadamard space $\cM$ \cite{afsari2011riemannian}. Then applying Lemma \ref{lma:a2}, we have
\$
&\sum_{i=1}^nd^2(u_t,x_{i,t+1})-\sum_{i=1}^nd^2(u_t,y_{i,t+1})\\
\leq& \sum_{i=1}^n\sum_{j=1}^nw_{ij}d^2(u_t,y_{j,t+1})-\sum_{i=1}^nd^2(u_t,y_{i,t+1}) =0,
\$
where the first inequality uses the definition \eqref{alg:II} of $\{x_{i,t+1}\}$ and the equality follows from $\sum_iw_{ij}=1$. Thus, combining the above analysis, we obtain that
\$
 \frac{1}{n}\sum_{i=1}^n\sum_{t=1}^T\left(d^2(u_t,x_{i,t})-d^2(u_t,y_{i,t+1})\right) 
    \leq\ &\frac{1}{n}\sum_{i=1}^nd^2(u_1,x_{i,1})+2D\sum_{t=1}^{T}d(u_{t+1},u_t)\\
    \leq\ &D^2+2D\sum_{t=1}^Td(u_{t+1},u_t),
\$
where the second inequality uses $d^2(u_1,x_{i,1})\leq D^2$. 

\section{Proof of Theorem \ref{thm:1}}
To prove Theorem \ref{thm:1}, we first establish an auxiliary lemma. 

\begin{lemma}\label{lma:d}
    Consider the same conditions used in Theorem \ref{thm:1}. Let $\{x_{i,t}\}$ be the local states generated by the decentralized algorithm in Section \ref{sec:2.3}, and $\{u_t\}$ a sequence of comparators in $\cX$. Then it holds that 
    \$
    &\frac{1}{n}\sum_{i=1}^n\sum_{t=1}^T\left(f_{i,t}(x_{i,t})-f_{i,t}(u_t)\right) 
    \leq \eta L^2\zeta(\kappa,D+\eta L)T+\frac{1}{\eta}\left(D^2+2DP_T\right),
    \$
    where $\zeta(\kappa,c)=\sqrt{|\kappa|}c\coth(\sqrt{|\kappa|}c)$, $\kappa$ is the lower sectional curvature bound of $\cM$, $D$ is the diameter of $\cX$,  $L$ is the Lipschitz constant of $f_{i,t}$, and $P_T=\sum_{t=1}^Td(u_{t+1},u_t)$.
\end{lemma}

\begin{proof}[Proof of Lemma \ref{lma:d}]
    Due to the geodesic convexity of $f_{i,t}$, we have
    \#
    f_{i,t}(x_{i,t})-f_{i,t}(u_t)&\leq \inner{-\nabla_{i,t}}{\Log_{x_{i,t}}u_{t}} =\frac{1}{\eta}\inner{\Log_{x_{i,t}}z_{i,t+1}}{\Log_{x_{i,t}}u_{t}},\label{equ:d1}
    \#
    where $\nabla_{i,t}=\nabla f_{i,t}(x_{i,t})$ and  $z_{i,t+1}=\Exp_{x_{i,t}}(-\eta\nabla_{i,t}).$ Then using Lemma \ref{lma:2} to the triangle with vertices $x_{i,t},z_{i,t+1},u_t$, we have
    \$
    &2\inner{\Log_{x_{i,t}}z_{i,t+1}}{\Log_{x_{i,t}}u_{t}}\\
    \leq\ & \zeta(\kappa,D+\eta L)d^2(x_{i,t},z_{i,t+1})+d^2(x_{i,t},u_{t})-d^2(z_{i,t+1},u_t),
    \$
    where $\zeta(\kappa,c)=\sqrt{|\kappa|}c\coth(\sqrt{|\kappa|}c)$ is an increasing function of $c>0$, $\kappa<0$ is the lower sectional curvature bound of $\cM$. In the above inequality, we use the fact that 
    \$
    d(z_{i,t+1},u_t)\leq d(z_{i,t+1},x_{i,t})+d(x_{i,t},u_t)\leq D+\eta L,
    \$
    where $D=\sup_{x,y\in\cX}d(x,y)$ denotes the diameter of $\cX$ and $L$ is the Lipschitz constant of $f_{i,t}$. Using the Lipschitz property of $f_{i,t}$, we have 
    \$
    d^2(x_{i,t},z_{i,t+1})=\eta^2\norm{\nabla_{i,t}}^2\leq \eta^2L^2.
    \$
    Since $y_{i,t+1}=\cP_{\cX}(z_{i,t+1})$ and $u_t\in\cX$, by Lemma \ref{lma:1}, we have
    \$
    d^2(y_{i,t+1},u_t)\leq d^2(z_{i,t+1},u_t). 
    \$
    Therefore, it holds that
    \$
    2\inner{\Log_{x_{i,t}}z_{i,t+1}}{\Log_{x_{i,t}}u_{t}} 
    \leq\ & \eta^2L^2\zeta(\kappa,D+\eta L)+d^2(x_{i,t},u_t)-d^2(y_{i,t+1},u_t).
    \$
    Summing up over $i$ and $t$ and applying Lemma \ref{lma:5}, we have 
    \$
    &\frac{2}{n}\sum_{i=1}^n\sum_{t=1}^T\inner{\Log_{x_{i,t}}z_{i,t+1}}{\Log_{x_{i,t}}u_{t}}\\
    \leq \ &\eta^2L^2\zeta(\kappa,D+\eta L)T+D^2+2D\sum_{t=1}^Td(u_{t+1},u_t).
    \$
    Then we conclude the proof by combining this inequality with \eqref{equ:d1}. 
\end{proof}

Then we are ready to prove Theorem \ref{thm:1}. 

\begin{proof}[Proof of Theorem \ref{thm:1}]

To upper bound the dynamic regret in \eqref{equ:1.2}, we start with the following decomposition
\$
f_{t}(x_{i,t})-f_t(u_t)=f_t(x_{i,t})-f_t(\bar x_t)+f_t(\bar x_t)-f_t(u_t),
\$
where $\bar x_t$ is the \Frechet mean of $\{x_{i,t}\}_{i=1}^n$. By the Lipschitz property of $f_{i,t}$ (and thus $f_t$) and Lemma \ref{lma:4}, we have
\#\label{equ:e1}
f_t(x_{i,t})-f_t(\bar x_t)&\leq Ld(x_{i,t},\bar x_t) \leq \frac{\eta\sigma_2(W)\sqrt{n}L^2}{1-\sigma_2(W)},
\#
where $\sigma_2(W)$ is the second largest singular value of $W$. Thus,
\$
 f_{t}(x_{i,t})-f_t(u_t) 
\leq\ &\frac{\eta\sigma_2(W)\sqrt{n}L^2}{1-\sigma_2(W)}+\frac{1}{n}\sum_{i=1}^n\left(f_{i,t}(\bar x_t)-f_{i,t}(u_t)\right)\\
\leq\ &\frac{2\eta\sigma_2(W)\sqrt{n}L^2}{1-\sigma_2(W)}+\frac{1}{n}\sum_{i=1}^n\left(f_{i,t}(x_{i,t})-f_{i,t}(u_t)\right),
\$
where the second inequality follows from the similar argument in \eqref{equ:e1} (Lipschitz property of $f_{i,t}$ and Lemma \ref{lma:4}). Now we sum this inequality over $i$ and $t$, we obtain that
\$
\textnormal{D-Regret}(\{x_{i,t}\}) 
\leq\ &\frac{2\eta\sigma_2(W)\sqrt{n}L^2T}{1-\sigma_2(W)}+\frac{1}{n}\sum_{i=1}^n\sum_{t=1}^T\left(f_{i,t}(x_{i,t})-f_{i,t}(u_t)\right)\\
\leq\ &\eta TC+\frac{1}{\eta}(D^2+2DP_T),
\$
where $C=L^2\zeta(\kappa,D+\eta L)+2\sigma_2(W)\sqrt{n}L^2/(1-\sigma_2(W))$ is a constant and $P_T$ denotes the path variation $\sum_{t=1}^Td(u_{t+1},u_t)$. 
\end{proof}

\section{Proof of Lemma \ref{lma:4.1}}

Recall that the consensus step \eqref{equ:4.1} maps points $\{y_{i,t+1}\}$ to $\{x_{i,t+1}\}$. Let $\bar y_{t+1}$ and $\bar x_{t+1}$ denote the \Frechet means of $\{y_{i,t+1}\}_{i=1}^n$ and $\{x_{i,t+1}\}_{i=1}^n$ as in \eqref{equ:ybar-t} and \eqref{equ:ybar-t+1}, respectively. Since $\cM$ is a Hadamard manifold with sectional curvatures lower bounded by $\kappa<0$, we can apply Lemma \ref{lma:2} to the geodesic triangle with vertices $y_{i,t+1},x_{i,t+1},\bar y_{t+1}$ and obtain that
\#
d^2(\bar y_{t+1}, x_{i,t+1})\notag
&\leq\zeta(\kappa,d(\bar y_{t+1},x_{i,t+1})) \cdot d^2(x_{i,t+1},y_{i,t+1})+d^2(\bar y_{t+1},y_{i,t+1})\\
&-2\inner{\Log_{y_{i,t+1}}x_{i,t+1}}{\Log_{y_{i,t+1}}\bar y_{t+1}},\label{equ:f1}
\#
where $\zeta(\kappa,c)=\sqrt{|\kappa|}c\coth(\sqrt{|\kappa|}c)$. Using definition \eqref{equ:4.1} of $x_{i,t+1}$, we have
\#\label{equ:f2}
\Log_{y_{i,t+1}}x_{i,t+1}=\gamma \sum_{j=1}^nw_{ij}\Log_{y_{i,t+1}}y_{j,t+1}
\#
and $d(y_{i,t+1},x_{i,t+1})\leq \gamma\max_{ij}d(y_{i,t+1},y_{j,t+1})\leq \gamma D$, where $D$ is the diameter of $\cX$. Therefore, $d(\bar y_{t+1},x_{i,t+1})\leq (1+\gamma)D$ by the triangle inequality, and
\#\label{equ:f3}
\zeta(\kappa,d(\bar y_{t+1},x_{i,t+1}))\leq \zeta(\kappa,(1+\gamma)D)\eqqcolon C_2,
\#
since $\zeta(\kappa,c)$ is an increasing function of $c$. Furthermore, by using \eqref{equ:f2}, we have
\#
&2\inner{\Log_{y_{i,t+1}}x_{i,t+1}}{\Log_{y_{i,t+1}}\bar y_{t+1}}\notag\\
=\ &\gamma\sum_{j=1}^nw_{ij}\cdot 2\inner{\Log_{y_{i,t+1}}y_{j,t+1}}{\Log_{y_{i,t+1}}\bar y_{t+1}}.\label{equ:f4}
\#
Applying Lemma \ref{lma:a1} to the geodesic triangle with vertices $y_{i,t+1},y_{j,t+1},\bar y_{t+1}$, we have 
\#
&2\inner{\Log_{y_{i,t+1}}y_{j,t+1}}{\Log_{y_{i,t+1}}\bar y_{t+1}}\notag\\
\geq\ & d^2(\bar y_{t+1},y_{i,t+1})+d^2(y_{i,t+1},y_{j,t+1})-d^2(\bar y_{t+1},y_{j,t+1}).\label{equ:f5}
\#
Now substituting \eqref{equ:f3}, \eqref{equ:f4}, and \eqref{equ:f5} into \eqref{equ:f1} and summing over $i$, we obtain 
\#
&\sum_{i=1}^nd^2(\bar y_{t+1},x_{i,t+1})\notag\\
\leq\ &C_2\cdot \sum_{i=1}^nd^2(x_{i,t+1},y_{i,t+1})+\sum_{i=1}^nd^2(\bar y_{t+1},y_{i,t+1})\notag\\
+\ &\gamma\sum_{i,j=1}^nw_{ij}\left(d^2(\bar y_{t+1},y_{j,t+1})-d^2(\bar y_{t+1},y_{i,t+1})-d^2(y_{i,t+1},y_{j,t+1})\right)\notag\\
=\ &C_2\cdot \sum_{i=1}^nd^2(x_{i,t+1},y_{i,t+1})+\sum_{i=1}^nd^2(\bar y_{t+1},y_{i,t+1})-\gamma\sum_{i,j=1}^nw_{ij}d^2(y_{i,t+1},y_{j,t+1}),\label{equ:f6}
\#
where the last equality uses the fact that $\sum_{i}w_{ij}=\sum_{j}w_{ij}=1$. By \eqref{equ:f2}, we have 
\$
d^2(x_{i,t+1},y_{i,t+1})&=\gamma^2\norm{\sum_{j=1}^nw_{ij}\Log_{y_{i,t+1}}y_{j,t+1}}^2\\
&\leq \gamma^2\sum_{j=1}^nw_{ij}\sum_{j=1}^nw_{ij}\norm{\Log_{y_{i,t+1}}y_{j,t+1}}^2\\
&=\gamma^2\sum_{j=1}^nw_{ij}d^2(y_{i,t+1},y_{j,t+1}),
\$
where the first inequality uses the Cauchy-Schwarz inequality and the second equality follows from the observation that  $\sum_{j}w_{ij}=1$. 
Substituting this into  \eqref{equ:f6}, we obtain that 
\$
\sum_{i=1}^nd^2(\bar y_{t+1},x_{i,t+1})\leq \sum_{i=1}^nd^2(\bar y_{t+1},y_{i,t+1})- (\gamma-C_2\gamma^2)\sum_{i,j=1}^nw_{ij}d^2(y_{i,t+1},y_{j,t+1}).
\$
Suppose $\gamma= 1/(2C_3)\leq 1$ where $C_3=\zeta(\kappa,2D)$. Then we have
\#\label{equ:f7}
\sum_{i=1}^nd^2(\bar y_{t+1},x_{i,t+1})\leq \sum_{i=1}^nd^2(\bar y_{t+1},y_{i,t+1})- \frac{1}{4C_3}\sum_{i,j=1}^nw_{ij}d^2(y_{i,t+1},y_{j,t+1}).
\#
Applying Lemma \ref{lma:a1} to the geodesic triangle with vertices $y_{i,t+1},y_{j,t+1},\bar y_{t+1}$, we have
\$
&d^2(y_{i,t+1},y_{j,t+1})\\
\geq\ &d^2(\bar y_{t+1},y_{i,t+1})+d^2(\bar y_{t+1},y_{j,t+1})-2\inner{\Log_{\bar y_{t+1}}y_{i,t+1}}{\Log_{\bar y_{t+1}}y_{j,t+1}}.
\$
Multiplying this with $w_{ij}$ and summing over $i,j$, we obtain that
\#
&\sum_{i,j=1}^nw_{ij}d^2(y_{i,t+1},y_{j,t+1})\notag\\
\geq\ &2\sum_{i=1}^nd^2(\bar y_{t+1},y_{i,t+1})-2\sum_{i,j=1}^nw_{ij}\inner{\Log_{\bar y_{t+1}}y_{i,t+1}}{\Log_{\bar y_{t+1}}y_{j,t+1}},\label{equ:f8}
\#
where we use the condition that $\sum_iw_{ij}=\sum_jw_{ij}=1$. Since $\bar y_{t+1}$ is the \Frechet mean of $\{y_{i,t+1}\}$, we have $\sum_{i}\Log_{\bar y_{t+1}}y_{i,t+1}=0.$
Thus, it follows that
\$
&\sum_{i,j=1}^nw_{ij}\inner{\Log_{\bar y_{t+1}}y_{i,t+1}}{\Log_{\bar y_{t+1}}y_{j,t+1}}\\
= &\sum_{i,j=1}^n(w_{ij}-1/n)\inner{\Log_{\bar y_{t+1}}y_{i,t+1}}{\Log_{\bar y_{t+1}}y_{j,t+1}}.
\$
Using the singular value analysis in the proof of Lemma \ref{lma:3}, we can show that
\$
\sum_{i,j=1}^n(w_{ij}-1/n)\inner{\Log_{\bar y_{t+1}}y_{i,t+1}}{\Log_{\bar y_{t+1}}y_{j,t+1}}\leq \sigma_2(W)\sum_{i=1}^nd^2(\bar y_{t+1},y_{i,t+1}),
\$
where $\sigma_2(W)=\sigma_1(W-1/n)$ is the second largest singular value of $W$. Substituting this inequality into \eqref{equ:f8}, we obtain that
\$
\sum_{i,j=1}^nw_{ij}d^2(y_{i,t+1},y_{j,t+1})\geq2(1-\sigma_2(W))\sum_{i=1}^nd^2(\bar y_{t+1},y_{i,t+1}).
\$
Combining this with \eqref{equ:f7}, we obtain that 
\$
\sum_{i=1}^nd^2(\bar y_{t+1},x_{i,t+1})\leq \left(1-\frac{1-\sigma_2(W)}{2C_3}\right)\sum_{i=1}^nd^2(\bar y_{t+1},y_{i,t+1}).
\$
Using the definition of $\bar x_{t+1}$, we have that
\$
\sum_{i=1}^nd^2(\bar x_{t+1},x_{i,t+1})\leq \sum_{i=1}^nd^2(\bar y_{t+1},x_{i,t+1})\leq \left(1-\frac{1-\sigma_2(W)}{2C_3}\right)\sum_{i=1}^nd^2(\bar y_{t+1},y_{i,t+1}),
\$
which concludes the proof.

\section{Proof of Lemma \ref{lma:4.2}}

The proof of this lemma bears an resemblance to that of Lemma \ref{lma:4}. Let $\VF(\{x_{i,t}\})$ denote the \Frechet variance of $\{x_{i,t}\}_{i=1}^n$. Then we have $\VF(\{x_{i,1}\})=0$ as we initialize $x_{i,1}=x_1$ at the same point. Define
\$
z_{i,t+1}=\Exp_{x_{i,t}}(-\eta\nabla_{i,t}),\quad \textnormal{where }\nabla_{i,t}=\nabla f_{i,t}(x_{i,t}).
\$
Then $y_{i,t+1}=\cP_{\cX}(z_{i,t+1})$. Let $\bar y_{t+1}$ and $\bar z_{t+1}$ be the \Frechet mean of the points $\{y_{i,t+1}\}$ and $\{z_{i,t+1}\}$, respectively. Then it holds that
\$
n\VF(\{y_{i,t+1}\})=\sum_{i=1}^nd^2(\bar y_{t+1},y_{i,t+1})\leq \sum_{i=1}^nd^2(\cP_{\cX}(\bar z_{t+1}),y_{i,t+1})\leq \sum_{i=1}^nd^2(\bar z_{t+1}, z_{i,t+1}),
\$
where the first inequality uses  definition of $\bar y_{t+1}$ and the second inequality uses Lemma \ref{lma:1}. Then using definition of $\bar z_{t+1}$, we have
\$
\sum_{i=1}^nd^2(\bar z_{t+1}, z_{i,t+1})\leq \sum_{i=1}^nd^2(\bar x_{t},z_{i,t+1}),
\$
where $\bar x_t$ is the \Frechet mean of $\{x_{i,t}\}$. Using the $L$-Lipschitz property of $f_{i,t}$, we can show that
\$
d(z_{i,t+1},x_{i,t})=\eta\norm{\nabla_{i,t}}\leq \eta L.
\$
Then using the triangle inequality, we can show that 
\$
\sqrt{\sum_{i=1}^nd^2(\bar x_{t},z_{i,t+1})}&\leq \sqrt{\sum_{i=1}^nd^2(\bar x_t,x_{i,t})}+\sqrt{\sum_{i=1}^nd^2(x_{i,t},z_{i,t+1})}\\
&\leq \sqrt{n\VF(\{x_{i,t}\})}+\eta\sqrt{n}L.
\$
Combining the above inequalities, we have
\#\label{equ:lma4-proof-1}
\sqrt{\VF(\{y_{i,t+1}\})}\leq \sqrt{\VF(\{x_{i,t}\})}+\eta L.
\#
Applying Lemma \ref{lma:4.1}, we have 
\#
\sqrt{\VF(\{x_{i,t+1}\})}&\leq \sqrt{1-\frac{1-\sigma_2(W)}{2C_3}}\sqrt{\VF(\{y_{i,t+1}\})}\notag\\
&\leq \left(1-\frac{1-\sigma_2(W)}{4C_3}\right)\sqrt{\VF(\{y_{i,t+1}\})},\label{equ:lma4-proof-2}
\#
where $C_3=\zeta(\kappa,2D)$ with $D$ being the diameter of $\cX$ and the second inequality uses the fact $\sqrt{1-x}\leq 1-x/2$ for any $x\in[0,1]$. Combining \eqref{equ:lma4-proof-1} with \eqref{equ:lma4-proof-2}, we have
\$
\sqrt{\VF(\{x_{i,t+1}\})}\leq \left(1-\frac{1-\sigma_2(W)}{4C_3}\right)\sqrt{\VF(\{x_{i,t}\})}+\eta L.
\$
Then using an induction argument, we can easily show that 
\$
\sqrt{\VF(\{x_{i,t}\})}\leq \frac{4C_3\eta L}{1-\sigma_2(W)}
\$
for all $t$. This concludes the proof by observing that $d(x_{i,t},\bar x_t)\leq \sqrt{n\VF(\{x_{i,t}\})}$.

\section{Proof of Lemma \ref{lma:4.3}}

The proof of this lemma is similar to that of Lemma \ref{lma:5}. First, we adopt the same decomposition
\#
d^2(u_t,x_{i,t})-d^2(u_t,y_{i,t+1})=\ &d^2(u_t,x_{i,t})-d^2(u_{t+1},x_{i,t+1})\notag\\
+\ &d^2(u_{t+1},x_{i,t+1})-d^2(u_{t},x_{i,t+1})\notag\\
+\ &d^2(u_{t},x_{i,t+1})-d^2(u_t,y_{i,t+1}).\label{equ:i1}
\#
For the second term, we use the triangle inequality to obtain that
\#\label{equ:i2}
d^2(u_{t+1},x_{i,t+1})-d^2(u_t,x_{i,t+1})\leq 4Dd(u_{t+1},u_t),
\#
where we use the observations that $d(x_{i,t+1},u_{t})\leq 2D$ and $d(x_{i,t+1},u_{t+1})\leq 2D$ with $D$ being the diameter of $\cX$. For the third term, we adopt a different proof. Instead of using the Jensen inequality, we apply Lemma \ref{lma:2} to the geodesic triangle with vertices $x_{i,t+1},y_{i,t+1},u_t$ and obtain that
\#
&d^2(u_t,x_{i,t+1})-d^2(u_t,y_{i,t+1})\notag\\
\leq\ &\zeta(\kappa,d(u_t,x_{i,t+1}))\cdot d^2(x_{i,t+1},y_{i,t+1})-2\inner{\Log_{y_{i,t+1}}x_{i,t+1}}{\Log_{y_{i,t+1}}u_t}\notag\\
\leq\ & \zeta(\kappa,2D)\cdot d^2(x_{i,t+1},y_{i,t+1})-2\inner{\Log_{y_{i,t+1}}x_{i,t+1}}{\Log_{y_{i,t+1}}u_t},\label{equ:i3}
\#
where $\zeta(\kappa,c)$ is defined in Lemma \ref{lma:2} and the second inequality uses the observation that $d(x_{i,t+1},u_t)\leq 2D$. Using the definition of $x_{i,t+1}$, we have that
\$
2\inner{\Log_{y_{i,t+1}}x_{i,t+1}}{\Log_{y_{i,t+1}}u_t}=\gamma\sum_{j}w_{ij}\cdot 2\inner{\Log_{y_{i,t+1}}y_{j,t+1}}{\Log_{y_{i,t+1}}u_t}.
\$
Applying Lemma \ref{lma:a1} to the geodesic triangles with vertices $y_{i,t+1},y_{j,t+1},u_t$, we have
\$
&2\inner{\Log_{y_{i,t+1}}x_{i,t+1}}{\Log_{y_{i,t+1}}u_t}\\
\geq\ & \gamma\sum_{j}w_{ij}\left(d^2(u_t,y_{i,t+1})+d^2(y_{i,t+1},y_{j,t+1})-d^2(u_{t},y_{j,t+1})\right).
\$
Summing this over $i$, we obtain that
\#
&2\sum_{i=1}^n\inner{\Log_{y_{i,t+1}}x_{i,t+1}}{\Log_{y_{i,t+1}}u_t}\notag\\
\geq\ & \gamma\sum_{i,j=1}^nw_{ij}\left(d^2(u_t,y_{i,t+1})+d^2(y_{i,t+1},y_{j,t+1})-d^2(u_{t},y_{j,t+1})\right)\notag\\
=\ &\gamma\sum_{i,j=1}^nw_{ij}d^2(y_{i,t+1},y_{j,t+1}),\label{equ:i4}
\#
where we use the fact that $\sum_{i}w_{ij}=\sum_jw_{ij}=1$. 
By the definition of $x_{i,t+1}$, we have
\$
d^2(x_{i,t+1},y_{i,t+1})=\gamma^2\norm{\sum_{j=1}^nw_{ij}\Log_{y_{i,t+1}}y_{j,t+1}}^2.
\$
Using the Cauchy-Schwarz inequality and the fact that $\sum_jw_{ij}=1$, we can show that
\$
d^2(x_{i,t+1},y_{i,t+1})\leq \gamma^2\sum_{j=1}^nw_{ij}d^2(y_{i,t+1},y_{j,t+1}).
\$
Summing this over $i$ and combining it with \eqref{equ:i3} and \eqref{equ:i4}, we obtain that
\#
&\sum_{i=1}^n\left(d^2(u_t,x_{i,t+1})-d^2(u_t,y_{i,t+1})\right)\notag\\
\leq\ &(\gamma^2\zeta(\kappa,2D)-\gamma)\sum_{i,j=1}^nw_{ij}d^2(y_{i,t+1},y_{j,t+1})\leq 0\label{equ:i6}
\#
where we use the condition that $\gamma=1/(2C_3)$ with $C_3=\zeta(\kappa,2D)$. Then by combining \eqref{equ:i1}, \eqref{equ:i2}, and \eqref{equ:i6} and summing over $i$ and $t$, we obtain that
\$
&\frac{1}{n}\sum_{i=1}^n\sum_{t=1}^T\left(d^2(u_t,x_{i,t})-d^2(u_t,y_{i,t+1})\right) \\
\leq\ &\frac{1}{n}\sum_{i=1}^nd^2(u_1,x_{i,1})+4DP_T\\
\leq\ &D^2+4DP_T,
\$
where $P_T=\sum_{t=1}^Td(u_{t+1},u_t)$ and the second inequality uses the fact that $d(u_1,x_{i,1})\leq D$. This concludes the proof.

\section{Proof of Theorem \ref{thm:2}}

The proof of this theorem is analogous to that of Theorem \ref{thm:1}. By the same argument in the proof of Theorem \ref{thm:1}, we have
\$
f_t(x_{i,t})-f_t(u_t)\leq \frac{8C_3\eta\sqrt{n}L^2}{1-\sigma_2(W)}+\frac{1}{n}\sum_{i=1}^n\left(f_{i,t}(x_{i,t})-f_{i,t}(u_t)\right),
\$
where we use Lemma \ref{lma:4.2} and $C_3,L,\sigma_2(W)$ are defined in Lemma \ref{lma:4.2}. Then we sum this over $i$ and $t$ and obtain that 
\#\label{equ:thm-2-1}
\textnormal{D-Regret}(\{x_{i,t}\})\leq \frac{8C_3\eta\sqrt{n}L^2T}{1-\sigma_2(W)}+\frac{1}{n}\sum_{i=1}^n\sum_{t=1}^T\left(f_{i,t}(x_{i,t})-f_{i,t}(u_t)\right).
\#
By the geodesic convexity of $f_{i,t}$, we have
\#
    f_{i,t}(x_{i,t})-f_{i,t}(u_t)&\leq \inner{-\nabla_{i,t}}{\Log_{x_{i,t}}u_{t}} =\frac{1}{\eta}\inner{\Log_{x_{i,t}}z_{i,t+1}}{\Log_{x_{i,t}}u_{t}},\label{equ:thm-2-2}
    \#
    where $\nabla_{i,t}=\nabla f_{i,t}(x_{i,t})$ and  $z_{i,t+1}=\Exp_{x_{i,t}}(-\eta\nabla_{i,t}).$ Then applying Lemma \ref{lma:2} to the geodesic triangle with vertices $x_{i,t},z_{i,t+1},u_t$, we have
    \$
    &2\inner{\Log_{x_{i,t}}z_{i,t+1}}{\Log_{x_{i,t}}u_{t}}\\
    \leq\ &\zeta(\kappa,d(z_{i,t+1},u_t))\cdot d^2(x_{i,t},z_{i,t+1})+d^2(x_{i,t},u_t)-d^2(z_{i,t+1},u_t),
    \$
    where $\zeta(\kappa,c)$ is defined in Lemma \ref{lma:2} and $\kappa<0$ is the lower sectional curvature bound of $\cM$. By the triangle inequality and the $L$-Lipschitz property of $f_{i,t}$, we can show that $d(z_{i,t+1},u_t)\leq 2D+L$ and $d^2(x_{i,t},z_{i,t+1})\leq \eta^2L^2$, where we use $\eta\leq 1$ in the first inequality. Here $D$ denotes the diameter of $\cX$. Then we have
    \#\label{equ:thm-2-3}
     2\inner{\Log_{x_{i,t}}z_{i,t+1}}{\Log_{x_{i,t}}u_{t}} 
    \leq \eta^2\zeta(\kappa,2D+L)L^2+d^2(x_{i,t},u_t)-d^2(y_{i,t+1},u_t),
    \#
    where we use the observation that $d(y_{i,t+1},u_t)\leq d(z_{i,t+1},u_t)$ due to Lemma \ref{lma:1}. By combining \eqref{equ:thm-2-2} and \eqref{equ:thm-2-3} and summing over $i$ and $t$, we have
    \$
    &\frac{2}{n}\sum_{i=1}^n\sum_{t=1}^T\left(f_{i,t}(x_{i,t})-f_{i,t}(u_t)\right)\\
    \leq\ &\eta\zeta(\kappa,2D+L)L^2T+\frac{1}{\eta n}\sum_{i=1}^n\sum_{t=1}^T\left(d^2(x_{i,t},u_t)-d^2(y_{i,t+1},u_t)\right)\\
    \leq\ &\eta\zeta(\kappa,2D+L)L^2T+\frac{1}{\eta}(D^2+4DP_T),
    \$
    where the second inequality uses Lemma \ref{lma:4.3} and $P_T=\sum_{t=1}^Td(u_{t+1},u_t)$. Substituting this into \eqref{equ:thm-2-1}, we shall conclude the proof. 

\begin{figure}
    \centering
    \includegraphics[width=0.5\linewidth]{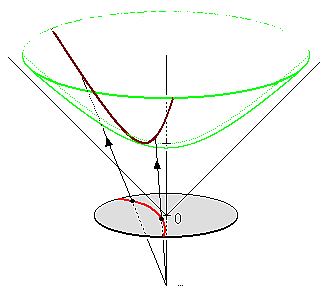}
    \caption{Projection from the hyperboloid to the \Poincare disk through the fixed point $(-1,0,0)$. The figure source is \href{https://en.wikipedia.org/wiki/Hyperboloid_model}{Wikipedia}.}
    \label{fig:a1}
\end{figure}

\section{Geometry}

\subsection{Hyperbolic spaces}\label{sec:a6}
This section provides a concise introduction to hyperbolic spaces. For a more comprehensive treatment of this topic, readers are directed to \cite{benedetti1992lectures}. Among various models to represent hyperbolic spaces, our experiments  primarily use two models: the hyperboloid model and the \Poincare disk model.

{\bf Hyperboloid model.}
 In the hyperboloid model, a hyperbolic space is modelled as the unit imaginary sphere in the Minkowski space. The Minkowski space is $\RR^{m+1}$ endowed with the Minkowski bilinear form 
    \$
    \left<x,y\right>_{M}=\sum_{i=1}^mx_iy_i-x_{0}y_{0}.
    \$
    The hyperboloid is the unit imaginary sphere given by
    \$
    \HH^m=\{x\in\RR^{m+1}\mid \inner{x}{x}_{M}=-1,x_{0}>0\}. 
    \$
    The tangent space to $\HH^m$ at $x$ is 
    \$
    T_{x}\HH^m=\{v\in\RR^{m+1}\mid \inner{x}{v}_{M}=0\}.
    \$
    The Minkowski bilinear form restricted to $T_{x}\HH^m$ is an inner product. This gives the Riemannian metric of $\HH^m$. Given $x\in\HH^m$ and $v\in T_x\HH^m$, the exponential map is
    \$
    \Exp_{x}(v)=\cosh(\norm{v})x+\sinh(\norm{v})v/v. 
    \$
    Based on this one can easily solve the logarithm map.
    
{\bf \Poincare disk model.}
 For $\HH^2$, we use the \Poincare disk model for visualization. It is a disk of unit radius $\{(t_1,t_2)\in\RR^2\mid t_1^2+t_2^2<1 \}$. It can be obtained by projecting the hyperboloid onto the disk at $t=0$ through the fixed point $(-1,0,0)$, as shown in Figure \ref{fig:a1}. Notably, the geodesics are either the diameter of the disk or Euclidean circles contained within the disk that are orthogonal to the boundary of the disk.

\subsection{Symmetric Positive Definite (SPD) Matrices}

The space of SPD matrices is a Hadamard manifold when endowed with the affine-invariant metric \cite{chen2024riemannian}. For $X\in{\rm SPD}(m)$, the tangent space to ${\rm SPD}(m)$ at $X$ is the space ${\rm Sym}(m)$ of all symmetric matrices in $\RR^{m\times m}$. For any $X,Y\in{\rm SPD}(m)$ and $V\in{\rm Sym}(m)$, the exponential map, logarithm map, and pairwise distance are given by \cite{moakher2005differential}
\$
\Exp_X(V)&=X^{1/2}\exp(X^{-1/2}VX^{-1/2})X^{1/2},\\
\Log_X(Y)&=X^{-1/2}\log(X^{1/2}YX^{1/2})X^{-1/2},\\
d(X,Y)&=\norm{\log(X^{-1/2}YX^{-1/2})}_s{\rF},
\$
where $\exp(\cdot)$ and $\log(\cdot)$ are the matrix exponential and matrix logarithm, respectively.

\bibliographystyle{unsrt}
\bibliography{references}

\end{document}